\documentclass[11pt, leqno]{amsart}
\usepackage{amsmath}
\usepackage{amsthm}
\usepackage{newtxtext, newtxmath}
\usepackage{mathrsfs}

\usepackage[colorinlistoftodos]{todonotes}
\renewcommand{\nu}{\upsilon}

\usepackage{graphicx}
\usepackage[font=small,labelfont=bf]{caption}
\usepackage{epstopdf}
\usepackage{xcolor}
\usepackage{float}
\usepackage{pgfplots}
\usepackage{listings}
\usepackage{longtable}
\usepackage{comment}
\usepackage{esint}
\usepackage{nicefrac}
\usepackage{tikz}
\usetikzlibrary{arrows, patterns, patterns.meta, calc, math}
\usepackage[normalem]{ulem} 
\usepackage[parfill]{parskip}
\usepackage{esint}
\usepackage{linegoal}

\usepackage{mathtools}
\usepackage[utf8]{inputenc}
\usepackage[T1]{fontenc}
\usepackage[shortlabels, inline]{enumitem}
\DeclarePairedDelimiterX\set[1]\lbrace\rbrace{#1}

\usepackage{scalerel}[2014/03/10]
\usepackage[usestackEOL]{stackengine}

\usepackage[backref=page]{hyperref}
\hypersetup{
plainpages=false,
colorlinks,
linkcolor={cyan!90!black},
citecolor={magenta},
urlcolor={red!90!black},
bookmarksdepth=2,} 
\usepackage[noadjust]{cite}

\hypersetup{
pdftitle={Boundedness of the parabolic Riesz Transform},
pdfsubject={Mathematics, PDE, Analysis},
pdfauthor={Khalid Baadi, Moritz Egert, Benjamin W. Kosmala},
pdfkeywords={}
}

\newtheorem{thm}{Theorem}[section]
\newtheorem{cor}[thm]{Corollary}
\newtheorem{prop}[thm]{Proposition}
\newtheorem{lem}[thm]{Lemma}

\theoremstyle{definition}
\newtheorem{defn}[thm]{Definition}

\theoremstyle{remark}
\newtheorem{rem}[thm]{Remark}

\definecolor{Khalid}{rgb}{0.8, 0.1, 0.1}
\definecolor{Benjamin}{rgb}{0.0, 0.5, 0.8}
\definecolor{Moritz}{rgb}{0.7, 0.1, 0.6}

\newcommand{\Z}{\mathcal{Z}}
\renewcommand{\L}{\mathrm{L}}

\newcommand{\Cont}{\mathrm{C}}
\newcommand{\e}{\mathrm{e}}
\newcommand{\ii}{\mathrm{i}}
\renewcommand{\d}{\mathrm{d}}
\newcommand{\dd}{\,\mathrm{d}}
\newcommand{\E}{\mathcal{E}}
\newcommand{\EE}{\mathrm{E}}

\newcommand{\R}{\mathbb{R}}
\newcommand{\C}{\mathbb{C}}
\newcommand{\N}{\mathbb{N}}
\newcommand{\D}{\mathbb{D}}
\newcommand{\norm}[1]{\Vert#1\Vert}
\newcommand{\abs}[1]{\vert#1\vert}
\newcommand{\one}{\boldsymbol{1}}
\newcommand{\F}{\mathcal{F}}
\newcommand{\M}{\mathrm{M}}
\renewcommand{\H}{\mathcal{H}}
\newcommand{\eps}{\varepsilon}

\DeclareMathOperator{\Avg}{Avg}
\DeclareMathOperator{\Div}{div}

\def\Xint#1{\mathchoice
{\XXint\displaystyle\textstyle{#1}}%
{\XXint\textstyle\scriptstyle{#1}}%
{\XXint\scriptstyle\scriptscriptstyle{#1}}%
{\XXint\scriptscriptstyle%
\scriptscriptstyle{#1}}%
\!\int}
\def\XXint#1#2#3{{\setbox0=\hbox{$#1{#2#3}{%
\int}$ }
\vcenter{\hbox{$#2#3$ }}\kern-.6\wd0}}
\def\barint{\,\Xint -} 
\def\bariint{\barint_{} \kern-.4em \barint}
\def\bariiint{\bariint_{} \kern-.4em \barint}

\newcommand{\supp}{\mathrm{supp}}

\newcommand{\dom}{\mathrm{dom}}
\newcommand{\RH}{\mathcal{R}_\H}
\renewcommand{\c}{\mathrm{c}}

\colorlet{shadecolor}{gray!20}
\pgfplotsset{compat=1.9}
\usepgflibrary{fpu}
\makeatletter
\let\c@equation\c@thm
\makeatother
\numberwithin{equation}{section}

\author{Khalid Baadi}
\address{Universit{\'e} Paris-Saclay, CNRS, Laboratoire de Math\'{e}matiques d'Orsay, 91405 Orsay, France}
\email{khalid.baadi@universite-paris-saclay.fr}
\author{Moritz Egert}
\address{TU Darmstadt, Fachbereich Mathematik, Schlossgartenstr.\ 7, 64289 Darmstadt, Germany}
\email{egert@mathematik.tu-darmstadt.de}
\author{Benjamin W. Kosmala}
\address{TU Darmstadt, Fachbereich Mathematik, Schlossgartenstr.\ 7, 64289 Darmstadt, Germany}
\email{kosmala@mathematik.tu-darmstadt.de}
\date{\today}
\title[$\L^p$ bounds for parabolic Riesz transforms]{$\L^p$ bounds for parabolic Riesz transforms with rough coefficients: The case $1<p \leq 2$}

\begin{document}
\begin{abstract}
We establish the first results on $\L^p$ bounds for Riesz transforms associated with non-autonomous second-order parabolic differential operators in divergence form with bounded coefficients that depend measurably on all variables. In the case of complex coefficients, we identify the maximal open range of exponents $1<p \leq2$ through the availability of uniform $\L^p$ resolvent bounds. This open range always contains the lower parabolic Sobolev conjugate of $2$ and the result is sharp in spatial dimension $n \geq 2$. For real coefficients, we prove extrapolation to the full range. Our argument relies on novel space-time off-diagonal bounds based on two complementary geometries: parabolic cubes on small scales and regions modeled on the half-order time derivative of a parabolic Bessel potential on large scales.
\end{abstract}

\maketitle

\setcounter{tocdepth}{1} 
\tableofcontents

\section{Introduction}

Riesz transforms are among the central objects of harmonic analysis. As prototypical examples of both Fourier multipliers and singular integral operators, they have played a pivotal role in the development of Calder\'on--Zygmund theory and continue to serve as a benchmark for many of its fundamental techniques.

Given a second-order operator $\mathcal{L}$ equipped with a suitable functional calculus and an underlying first-order differential structure $\D$, the associated Riesz transform is defined by
\begin{equation*}
     \mathcal{R}_{\mathcal{L}} \coloneqq \D\mathcal{L}^{-1/2}.
\end{equation*}
This abstract viewpoint encompasses a broad class of operators, including generalized Laplacians on Lie groups, graphs, Riemannian manifolds, and subsets of Euclidean space with boundary conditions, but also Schr\"odinger-type operators.

In this paper, we study the case in which $\mathcal{L} = \H$ is a heat-type operator:
\begin{align*}
    \H u =\partial_t u-\Div_x(A(x,t)\nabla_x u) \quad \text{on } \R^{n+1}
\end{align*}
is the sum of a time derivative and a non-autonomous elliptic part in divergence form. The parabolic first-order derivative $\D = (\nabla_x, \smash{D_t^{1/2}})$ is of fractional order $\nicefrac{1}{2}$ in time, turning $\RH$ into the composition of two non-local operators. The coefficients $A$ are bounded, measurable and elliptic (see Section~\ref{subsec:parabolicOP}), and $\H$ can be realized as an $m$-accretive operator on $\L^2(\R^{n+1})$ by Kaplan's trick~\cite{kaplan1966abstract}. As our main result, we establish the optimal open range of exponents $1 < p \leq2$ such that $\RH$ extends to a bounded operator on $\L^p$ or, equivalently, that the \emph{a priori} estimate
\begin{align*}
    c\bigl(\|\nabla_x u\|_p + \|D_t^{1/2} u\|_p \bigr) \leq \|\H^{1/2} u\|_p
\end{align*}
holds with $c>0$ independent of $u$ whenever $u$ belongs to the energy space for $\H$.
\subsection{History}

The theory originates in the Euclidean setting, where Riesz transforms are associated with the Laplacian on $\R^n$ and bounded on $\L^2$ by spectral theory. Calder\'on and Zygmund~\cite{CZ} were the first to prove boundedness on $\L^p$ for all $p\in(1,2]$, while duality extends the result to all $p\in(1,\infty)$. Boundedness on $\L^p$, in particular in the range $1<p\leq2$ and beyond, has since been pursued in a wide variety of settings~\cite{Str83,CD99,DM99,S99,Rus00,Rus01,ACDH04,AC05,AMR07,SV08,BF16,CCFR17,BEH20,LMS+23,Fen26, Egert_Riesz, Bechtel_Riesz}.

For divergence-form elliptic operators with rough coefficients, the lack of smoothness gives rise to fundamentally new phenomena. For such operators, even the $\L^2$ theory is a consequence of the solution of the Kato problem \cite{AHLMT02}. $\L^p$ bounds are typically restricted to a subrange of exponents $1<p<2$ and duality is no longer applicable to treat the case $2<p<\infty$. This phenomenon of limited-range extrapolation beyond Calderón--Zygmund theory was first explored in the seminal works of Blunck--Kunstmann and Hofmann--Martell \cite{blunck2003calderon,HM03}; see also the monographs \cite{auscher2007necessary,AEbook2023} for comprehensive historical accounts. The resulting theory has since been extended to degenerate elliptic operators \cite{CR15,ARR15,cruz2017kato}, Schr\"odinger operators~\cite{Morris-Dumont} and generalized Stokes operators~\cite{HT26}.

For special cases of parabolic operators, assuming enough temporal regularity on the coefficients of the elliptic part, Nystr\"om~\cite{Nystr16} and Ouhabaz~\cite{Ouhabaz21} proved by different techniques that the parabolic Riesz transform is bounded on $\L^2$. The full parabolic Kato problem for merely measurable coefficients was solved in \cite{AEN2020}. Beyond the Hilbert space setting, the only available result on parabolic Riesz transforms is the $\L^p$ boundedness obtained in \cite{Ouhabaz21} for autonomous operators with real coefficients, based on interpolation arguments and maximal regularity for the associated abstract Cauchy problem. Here, we establish $\L^p$ boundedness of parabolic Riesz transforms in the non-autonomous case with merely measurable coefficients depending on all variables.

\subsection{Main ideas and contributions}

For the classical heat operator $\partial_t-\Delta_x$, the Riesz transform
\begin{align*}
    \RH = \D (\partial_t - \Delta_x)^{-1/2}
\end{align*}
is the Fourier multiplier with symbol
\begin{align*}
    \biggl(\frac{\ii \xi}{\sqrt{\ii \tau + |\xi|^2}},\frac{|\tau|^{1/2}}{\sqrt{\ii \tau + |\xi|^2}}\biggr),
\end{align*}
whose second component is singular along the hyperplane $\{\tau=0\}$. Its $\L^p$ boundedness therefore relies on techniques that separate the time and space variables, such as the Marcinkiewicz multiplier theorem. For generalized heat operators, however, the measurable coefficients couple space and time, rendering such arguments unavailable. Nevertheless, we show that enough separation of the time and space variables can still be recovered to replace the classical coordinate-wise arguments. The remainder of this section gives an informal overview of the ideas underlying this approach.

We start from the Calder\'on reproducing formula
\begin{align}
\label{eq:intro_kernel}
    \RH u = \frac{2}{\pi} \int_0^\infty \lambda \D \E_\lambda u \; \frac{\mathrm{d} \lambda}{\lambda},
\end{align}
which represents $\RH=\D\H^{-1/2}$ as an integral operator with operator-valued kernel in the additional parameter $\lambda>0$. Here,
$\E_\lambda := (1+\lambda^2 \H)^{-1}$. Following Auscher~\cite{auscher2007necessary}, $\RH$ should be bounded on $\L^p$ precisely when $(\lambda \D \E_\lambda)_{\lambda>0}$ is uniformly bounded on $\L^p$. To understand the mechanism in more detail, we first return to the model case $\H=\partial_t-\Delta_x$. 

For this special operator, the temporal component of $\lambda\D\E_\lambda$ has an explicit convolution kernel $K_\lambda(x,t)$, namely the half-order time derivative of a parabolic Bessel kernel, satisfying
\begin{align}\label{eq: Kernel estimates Laplacian}
    |K_\lambda(x,t)| \leq \lambda^{-(n+3)}\bigg(\frac{\lambda^2}{\abs{t}}\bigg)^{3/2} \e^{-c \frac{\abs{x}}{\lambda}}
\end{align}
for $t<0$ and $|x|\ge \lambda$. The decay in the time variable is insufficient for all extrapolation techniques in parabolic scaling that we are aware of, since $\nicefrac{3}{2}$ is below the homogeneous dimension $n+2$. The driving idea is thus to combine two different geometries: parabolic cubes, reflecting the local geometry of $\H$, on small scales, and regions modeled on the level sets of the half-order time derivative of the parabolic Bessel kernel at large scales. See Figure~\ref{fig: time and space supports intro} for an illustration.

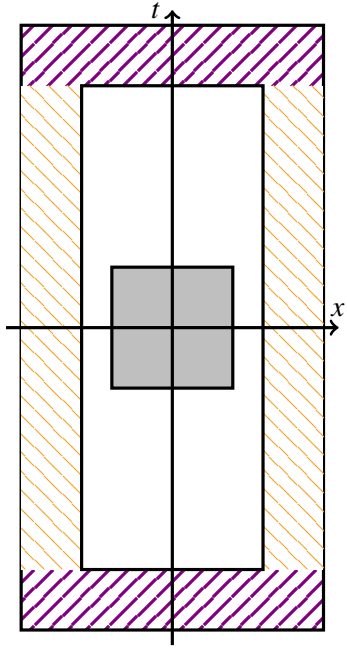
\begin{figure}[ht]
    \definecolor{gold}{HTML}{F1A226}
    \colorlet{color1}{violet}
    \colorlet{color2}{gold}
    \begin{tikzpicture}[scale=0.4]
        \pgfdeclarelayer{layer1}
        \pgfdeclarelayer{layer2}
        \pgfdeclarelayer{layer3}
        \pgfdeclarelayer{layer4}
        \pgfdeclarelayer{layer5}
        \pgfsetlayers{layer1,layer2,layer3,layer4,layer5}
        
        \tikzmath{\Ballx=2; \Ballt=2; \InnerAnnulusx=3; \InnerAnnulust=8; \OuterAnnulusx=5; \OuterAnnulust=10;}
        
        \begin{pgfonlayer}{layer1}
            \filldraw[very thick, pattern={Lines[angle=45,distance=5pt]}, pattern color=color1]
            (-\OuterAnnulusx,-\OuterAnnulust)
            rectangle (\OuterAnnulusx,\OuterAnnulust);
        \end{pgfonlayer}
        
        \begin{pgfonlayer}{layer2}
            \fill[fill=white]
            ({-\OuterAnnulusx+0.02},-\InnerAnnulust)
            rectangle (\OuterAnnulusx-0.02,\InnerAnnulust);
        \end{pgfonlayer}
        
        \begin{pgfonlayer}{layer3}
            \fill[pattern={Lines[angle=-45,distance=5pt]}, pattern color=color2]
            (-\OuterAnnulusx,-\InnerAnnulust)
            rectangle (\OuterAnnulusx,\InnerAnnulust);
        \end{pgfonlayer}
        
        \begin{pgfonlayer}{layer4}
            \filldraw[very thick, fill=white]
            (-\InnerAnnulusx,-\InnerAnnulust)
            rectangle (\InnerAnnulusx,\InnerAnnulust);
        \end{pgfonlayer}
        
        \begin{pgfonlayer}{layer5}
            \filldraw[very thick, fill=lightgray]
            (-\Ballx,-\Ballt)
            rectangle (\Ballx,\Ballt);

            \draw[black, very thick, ->] ({-\OuterAnnulusx-0.5},0) -- ({\OuterAnnulusx+0.5},0) node[above] {$x$};
            \draw[black, very thick, ->] (0,{-\OuterAnnulust-0.5}) -- (0,{\OuterAnnulust+0.5}) node[left] {$t$};
            
        \end{pgfonlayer}
    \end{tikzpicture}
    \caption{Change of geometry on large scales: The colored annulus is separated from the centered parabolic cube (\textcolor{gray}{gray}) of radius $r$ by distances of $N^{2j} r^2$ in time (\textcolor{color1}{violet}) and $2^jr$ in space (\textcolor{color2}{golden}). Typically, $N\geq 2^n$, so that the right-hand side in \eqref{eq: Kernel estimates Laplacian} is controlled by $2^{-j\varepsilon}$ with $\varepsilon>1$ on the full annulus.}
\label{fig: time and space supports intro}
\end{figure}

Relative to this two-scale geometry, we recover sufficient off-diagonal decay of Gaffney type. In this way, the change of geometry allows us to use the exponential decay in the spatial variable in order to gain decay in the time variable. This geometric insight already appears implicitly in \cite{AEN2020} and is developed here into a systematic extrapolation framework.

\subsection{Main results and strategy}

Following \cite{auscher2007necessary,AEbook2023}, we introduce the limiting exponents
\begin{align*}
    & p_-(\mathcal{H})\coloneqq\inf\left\{ p\ge1: \  \left ( \mathcal{E}_\lambda \right )_{\lambda>0} \ \text{is (uniformly)} \ \L^p \ \text{bounded} \right\},\\
    & q_-(\mathcal{H})\coloneqq\inf\left\{ p\ge1: \ \left (\lambda \mathbb{D} \mathcal{E}_\lambda \right )_{\lambda>0} \ \text{is (uniformly)} \ \L^p \ \text{bounded} \right\}.
\end{align*}
The preceding discussion suggests that $q_-(\H)$ governs the range of exponents for which $\RH$ is bounded on $\L^p$.

Even for real-valued coefficients, however, the value of $q_-(\H)$ is far from obvious, whereas for the more accessible exponent $p_-(\H)$ related to $\L^p$ resolvent bounds we have $p_-(\H)=1$ by the Gaussian heat kernel estimates. This makes the identification of both exponents a central question that is also solved in the following main theorem of our paper.

\begin{thm}\label{thm: Théorème principal}
The following assertions hold.
\begin{enumerate}[(1)]
    \item \textbf{Range of exponents:} We have $p_-(\H)\in [1,2_\star)$, where $2_\star := \frac{2d}{d+2}\in (1,2)$ and $d=n+2$.
    
    \item \textbf{Sufficient condition for boundedness:} For every $p\in (p_-(\H),2]$, the parabolic Riesz transform
    \begin{equation*}
        \mathcal{R}_{\H} = (\nabla_x \H^{-1/2}, D_t^{1/2}\H^{-1/2})
    \end{equation*}
    extends to a bounded operator on $\L^p(\R^{n+1})$.
    
    \item \textbf{Necessary condition for boundedness:} If $p\in (1,2)$ is such that the parabolic Riesz transform $\mathcal{R}_{\H}$ extends to a bounded operator on $\L^p(\R^{n+1})$, then $p\geq p_-(\H)$.
    
    \item \textbf{Case of real coefficients:} If $A$ has real-valued coefficients, then $p_-(\H)=1$, and the spatial gradient component $\nabla_x \H^{-1/2}$ is of weak type $(1,1)$.
    
    \item \textbf{Sharpness within the class of parabolic operators:} Assume that $n\ge 2$. For every $\varepsilon>0$ there exist coefficients $A_\varepsilon$ such that for $\H_\varepsilon=\partial_t-\mathrm{div}_x(A_\varepsilon \nabla_x)$ we have
\begin{equation*}
    2_\star-\varepsilon \leq p_-(\H_\varepsilon).
\end{equation*}
\end{enumerate}
\end{thm}

Theorem~\ref{thm: Théorème principal} is an extrapolation result from the case $p=2$, corresponding to the parabolic Kato square root estimate \cite{AEN2020,AEN2025}. We conclude the introduction with an outline of the main steps of the proof and the overall structure of our paper.

\textbf{Space-time off-diagonal estimates:}
In Section~\ref{S4}, we establish sufficiently strong space-time off-diagonal estimates for $(\lambda \D \E_\lambda)_{\lambda>0}$, following the geometric idea outlined above. In the setting of Figure~\ref{fig: time and space supports intro}, denoting by $E$ the gray cube and by $F$ the colored annulus, we prove in Theorem~\ref{thm: parabolicODEs} the estimate
\begin{align*}
    \|\one_F \lambda \D\mathcal{E}_\lambda \one_E u \|_{p}
        \leq C\biggl(\frac{\lambda}{r} +\Bigl(\frac{\lambda}{r}\Bigr)^{4N}\biggr)
        N^{-j \varepsilon}
        \|\one_E u\|_{p},
\end{align*}
where $\varepsilon = 1+\nicefrac{1}{(1+p')}>1$, whenever $\H$ satisfies $\L^p$ resolvent estimates. In contrast to previous $\L^p$ off-diagonal estimates, see for instance \cite{auscher2007necessary,blunck2003calderon,HM03,AEbook2023}, ours is genuinely an $\L^p$ estimate and is not obtained by interpolation with an $\L^2$ energy estimate. Indeed, because of the non-local fractional derivative $D_t^{1/2}$, even the underlying $\L^2$ energy estimate is of comparable form and interpolation would fail to provide sufficient decay whenever $p<2_\star$.

\textbf{Extrapolation for operators with limited space-time decay:}
In Theorem~\ref{thm: BK} we prove a general weak type criterion for sublinear operators acting on functions of two variables tailored to our time-stretched annuli and the corresponding off-diagonal estimates above. This new $\L^p$ extrapolation framework, developed in Section~\ref{S5}, falls outside the scope of the Calder\'on--Zygmund and Blunck--Kunstmann theory on spaces of homogeneous type~\cite{CZ,Stein1993_HA,blunck2003calderon}, since no fixed metric is compatible with both the parabolic cubes and the time-stretched annuli.

\textbf{Critical exponents:}
In Section~\ref{S6} we prove that $p_-(\H)=q_-(\H)\in[1,2_\star)$ by a bootstrap argument along the parabolic Sobolev conjugates that successively compensates for the limited decay in the time variable in each step. This establishes part~(1) of Theorem~\ref{thm: Théorème principal}.

\textbf{Riesz transform bounds:}
In Section~\ref{S7} we combine the extrapolation framework of Section~\ref{S5} with the space-time off-diagonal estimates to establish boundedness of the parabolic Riesz transform in the optimal exponent range $(p_-(\H),2]$. Here, a single extrapolation step is not sufficient: the available kernel decay on $\L^2$ would only yield boundedness for $p>2_\star$, which in particular falls short in the case of real-valued coefficients. The key observation is that our space-time off-diagonal estimates are $p$-sensitive in the sense that we can iteratively restart the extrapolation argument from every exponent $q$ for which the Riesz transform is already known to be bounded, while the extrapolation interval relative to $q$ remains independent of $q$.

\textbf{The remaining parts of Theorem~\ref{thm: Théorème principal}:} In Section~\ref{S8} we treat the case of real coefficients, and more generally the case where Gaussian upper bounds for the heat kernel are available. The sharpness result, part~(5), is established in Section~\ref{S9} and relies on Mooney's irregular solution~\cite{Mooney2021}.

\textbf{Outlook and open problems:}
The techniques developed here are expected to be useful for developing limited-range Calder\'on--Zygmund theory for other non-local operators such as generalized Stokes operators~\cite{HT26}. We discuss several related open problems in Section~\ref{S10}.

\subsection{Notation}

Most of our notation is standard. Additionally, we shall make use of the following conventions.
\begin{enumerate}[label=$\blacklozenge$]
    \item For suitable exponents $p \in [1, \infty]$ we define conjugate indices by $\nicefrac{1}{p'} = 1-\nicefrac{1}{p}$ (H\"older), $\nicefrac{1}{p^\star}=\nicefrac{1}{p}-\nicefrac{1}{n+2}$ (upper Sobolev), $\nicefrac{1}{p_\star}=\nicefrac{1}{p}+\nicefrac{1}{n+2}$ (lower Sobolev).
    
    \item For $p,q\in[1,\infty]$ we define the interpolating index by $\nicefrac{1}{[p,q]_\theta}\coloneqq \nicefrac{(1-\theta)}{p} + \nicefrac{\theta}{q}$ for any $\theta \in [0,1]$ and the Sobolev gap $\gamma_{p,q}\coloneqq  |\nicefrac{(n+2)}{q} - \nicefrac{(n+2)}{p}|$.

    \item For $(x,t),(y,s)\in \R^{n+1}$, we denote their parabolic distance by $$\d((x,t),(y,s))=\max\big(|x-y|_\infty,\sqrt{|t-s|}\big).$$
    
    \item For $(x,t)\in\R^{n+1}$ and $r>0$, let $Q_r(x)$ denote the cube centered at $x$ with radius $r$ and sides parallel to the coordinate axes, and set $I_r(t)\coloneqq (t-r^2,t+r^2)$. The parabolic cube centered at $(x,t)$ with radius $r$ is $\Delta_r(x,t)\coloneqq Q_r(x)\times I_r(t)$. Depending on the context, we may omit the centers.
    \item For a parabolic cube $\Delta_r:=Q_r\times I_r \subset \mathbb{R}^{n+1}$ and $a,b>0$, we define the stretched cube $aQ_r\times bI_r:=Q_{ar}\times I_{br}$. For $N>1$, we set $C^N_1(\Delta_r):=4Q_r \times N^2 I_r$ and for all $j\ge 2$
    \begin{equation*}
    C^N_j(\Delta_r):= \left ( 2^{j+1}Q_r \times N^{j+1} I_r \right ) \setminus \left ( 2^{j}Q_r \times N^{j} I_r \right ).
    \end{equation*}
 
    \item Constants $C, c$ appearing in statements are always strictly positive and finite.
\end{enumerate}

\subsubsection*{\textbf{Acknowledgements}}

The first author acknowledges the support of a public grant from the Fondation Mathématique Jacques Hadamard through the Programme Visibilité Scientifique Junior FMJH, as well as the guidance of his PhD advisor, Professor Pascal Auscher, that made a first research stay in May and June 2025 possible. He also acknowledges his support for a second research stay of one week in December 2025. The first author would also like to express his warm thanks to the second and third authors, as well as to the Department of Mathematics at TU Darmstadt, where the ideas of this project took shape, for their very kind hospitality. The second author is grateful to Hendrik Vogt for sharing his thoughts on Blunck and Kunstmann's criterion, which led to the formulation of Theorem~\ref{thm: BK}. The first and the second authors, but not the third, express warm thanks to Giulio Mollo for the espresso machine in Moritz's bibliothèque.

\section{The parabolic operator \texorpdfstring{$\mathcal{H}$}{}: definition and the \texorpdfstring{ $\L^2$}{} theory}

This section gives a brief summary of the known $\L^2$ theory for parabolic operators in divergence form.

\subsection{Parabolic energy space}
We denote by $\F$ the Fourier transform with respect to the time variable $t$ and write $\tau$ for the corresponding frequency variable. Recall that if $u \in \L^2(\R^{n+1})$, then $u(x,\cdot) \in \L^2(\R)$ for almost every $x \in \R^n$ by Fubini's theorem. The expression
\begin{equation*}
    H_t u \coloneqq \F^{-1}\left(i\,\tfrac{\tau}{|\tau|}\,\F u\right)
\end{equation*}
defines the Hilbert transform. If $|\tau|^{1/2}\F u \in \L^2(\R^{n+1})$, the half-order time derivative is defined by
\begin{equation*}
    D_t^{1/2} u \coloneqq \F^{-1}\left(|\tau|^{1/2}\F u\right).
\end{equation*}
The following material is all taken from \cite{AEN2025} with the constant weight $\omega=1$. We define the parabolic energy space as 
\begin{equation*}
    \mathrm{E}\coloneqq \left\{u\in \L^2(\R^{n+1}) \, : \, \nabla_xu, \, D_t^{1/2}u \in \L^2(\R^{n+1})  \right\},
\end{equation*}
where $\nabla_x$ denotes the (distributional) gradient with respect to the spatial variables $x$. For $u \in \mathrm{E}$, the parabolic gradient is defined by $\D u \coloneqq (\nabla_x u, D_t^{1/2} u)$. We equip $\mathrm{E}$ with the norm $\|u\|_\mathrm{E} \coloneqq (\|u\|_2^2 + \|\nabla_x u\|_2^2 + \|D_t^{1/2} u\|_2^2)^{1/2}$, which makes $\mathrm{E}$ a Hilbert space containing $\Cont_0^\infty(\R^{n+1})$ as a dense subspace. Moreover, multiplication by functions in $\Cont_\mathrm{b}^1(\R^{n+1})$ defines bounded operators on $\mathrm{E}$. In particular, $(\mathrm{E}, \L^2(\R^{n+1}), \mathrm{E}^\star)$ is a Gelfand triple, where $\mathrm{E}^\star$ is the anti-dual of $\mathrm{E}$ and we have bounded operators
\begin{align*}
    \nabla_x  : \, \mathrm{E} \longrightarrow (\L^2(\R^{n+1}))^n, \quad D_t^{1/2}  : \, \mathrm{E} \longrightarrow \L^2(\R^{n+1}), \quad
    \D  : \, \mathrm{E} \longrightarrow (\L^2(\R^{n+1}))^{n+1},
\end{align*}
and their adjoints
\begin{align*}
    -\Div_x : \, \L^2(\R^{n+1})^n \longrightarrow \mathrm{E}^\star, \quad
    D_t^{1/2} : \, \L^2(\R^{n+1}) \longrightarrow \mathrm{E}^\star.
\end{align*}
From \cite[§12.1]{Kilbas} we recall a representation formula for $D_t^{1/2}$: For $u \in \mathcal{S}(\mathbb{R}^{n+1})$ we have
\begin{align}
\label{eq: Dt representation}
    D_t^{1/2} u(x,t) = \frac{1}{2\sqrt{2\pi}} \int_{\mathbb{R}} \frac{u(x,t) - u(x,s)}{|t-s|^{3/2}} \, \mathrm{d}s.
\end{align}

\subsection{The parabolic operator \texorpdfstring{$\mathcal{H}$}{} and its associated parabolic Riesz transform}
\label{subsec:parabolicOP}

We fix a matrix-valued function $A: \R^{n+1} \to M_n(\C)$ with  measurable complex-valued coefficients, satisfying
\begin{equation}\label{eq: ellipticité A}
\left| A(x,t) \xi \cdot \zeta \right| \leq M |\xi|\,|\zeta|, \quad 
\nu |\xi|^2 \leq \mathrm{Re}(A(x,t) \xi \cdot \overline{\xi})
\end{equation}
for some $M, \nu > 0$ and for all $\xi, \zeta \in \C^n$ and $(x,t) \in \R^n \times \R$. We define the parabolic operator $\H: \mathrm{E} \rightarrow \mathrm{E}^\star$ by setting
\begin{align*}
    \H(u)(v) \coloneqq \iint_{\R^{n+1}} H_t D_t^{1/2} u(x,t) \cdot \overline{D_t^{1/2} v(x,t)} + A(x,t) \nabla_x u(x,t) \cdot \overline{\nabla_x v(x,t)} \, \, \mathrm{d}(x,t).
\end{align*}
We formally write $\mathcal{H} = \partial_t - \Div_x(A \nabla_x)$.
Likewise, we define $\mathcal{H}^\star : \mathrm{E} \to \mathrm{E}^\star$ as the adjoint of $\mathcal{H}$ associated with the adjoint form defining $\mathcal{H}$. Formally, it is given by $\mathcal{H}^\star \coloneqq -\partial_t - \Div_x\!\left(A^\star \nabla_x\right)$, where $A^\star$ denotes the Hermitian adjoint of the matrix $A$. This operator is not in the same class as $\H$ but it is similar to such an operator via conjugation with the reversal in time $(\tau u)(x,t) \coloneqq u(x,-t)$. Hence, all operator bounds for $\H$ in Lebesgue spaces also hold for $\H^\star$ and vice versa. We will use this fact freely throughout the paper.

The following theorem summarizes the  $\L^2$ theory for the square root of $\H$.

\begin{thm}[{\unskip\cite[Theorem 1.1]{AEN2025}, \cite[Theorem 2.6]{AEN2020}}]\label{thm: théorie L2}
The maximal restriction of $\H$ to  $\L^2(\R^{n+1})$ is an $m$-accretive and injective operator on  $\L^2(\R^{n+1})$, and the domain of its unique maximal accretive square root is $\mathrm{E}$. Moreover, there exists a constant $C$ such that
\begin{equation}
  \frac{1}{C} \|\D u\|_2 \le \|\sqrt{\H} u\|_2 \le C \|\D u\|_2, \qquad (u \in \mathrm{E}).
\end{equation}
\end{thm}

This result shows that the operator $\mathbb{D}\mathcal{H}^{-1/2}$, defined on $\operatorname{ran}(\sqrt{\mathcal{H}})$, is of strong type $(2,2)$. The following lemma gives more precise information on $\operatorname{ran}(\sqrt{\mathcal{H}})$.

\begin{lem}\label{lem: many functions}
    We have $\L^{2_\star}(\R^{n+1}) \cap \L^2(\R^{n+1}) \subseteq\operatorname{ran}(\sqrt{\mathcal{H}})$.
\end{lem}

\begin{proof}
We will prove the lemma for $\H^\star$ in place of $\H$, which belongs to the same class of operators. Since $\H$ is injective, we have $(\sqrt{\H^\star})^{-1} = (\H^\star)^{-\nicefrac 12} = (\H^{-\nicefrac 12})^\star$ as closed operators in the sectorial functional calculi for $\H$ and $\H^\star$, see, e.g., \cite{egert2024harmonic}. Thus, the claim is that $\L^{2_\star}(\R^{n+1}) \cap \L^2(\R^{n+1}) \subseteq\dom((\H^{-\nicefrac 12})^\star)$.
		
So let $u \in \L^{2_\star}(\R^{n+1}) \cap \L^2(\R^{n+1})$. The parabolic Sobolev embedding Lemma \ref{lem: Sobolev} and the $\L^2$ boundedness of the Riesz transform produce a constant $C$ such that
\begin{align*}
	\abs{\langle u , \H^{-\nicefrac{1}{2}} v \rangle} \leq \norm{u}_{2_\star} \norm{\H^{-\nicefrac{1}{2}} v}_{2^\star} \leq C \norm{u}_{2_\star} \norm{\D \H^{-\nicefrac{1}{2}} v}_2 \leq C \norm{u}_{2_\star} \norm{v}_2
\end{align*}
holds for all $v \in \dom(\H^{-\nicefrac{1}{2}})$ and the claim follows.     
\end{proof}

In particular, $\mathbb{D}\mathcal{H}^{-1/2}$ is defined on $\mathcal{S}(\R^{n+1})$. By density, it has a unique bounded extension to an operator from $\L^2(\R^{n+1})$ to $\L^2(\R^{n+1})^{n+1}$. This extension is the \emph{parabolic Riesz transform associated with $\mathcal{H}$}, and we denote it by $\RH$. 

\subsection{Resolvent estimates: known results}

For $\lambda > 0$, we set
\begin{equation}
    \E_\lambda \coloneqq (1+\lambda^2 \H)^{-1}, \quad \E_\lambda^\star \coloneqq (1+\lambda^2 \H^\star)^{-1}.
\end{equation}
By $m$-accretivity, they are uniformly bounded on $\L^2(\R^{n+1})$ and adjoint to each other. We recall the following key result on uniform boundedness and off-diagonal estimates in $\L^2$.

\begin{prop}[{\cite[Lemma 4.4]{AEN2025}}]\label{prop: classical ODEs}
There are constants $C$ and $c$ such that for all $\lambda>0$ and $u\in \L^2(\R^{n+1})$ we have 
\begin{align}\label{eq: Classical L2 ODES}
    \norm{\one_F \E_\lambda \one_E u}_2 + \norm{\one_F \lambda \nabla_x \E_\lambda \one_E u}_2 \leq C \e^{-c\frac{\d(E,F)}{\lambda}} \norm{\one_E u}_2,
\end{align}
for all measurable sets $E,F \subseteq \R^{n+1}$ and
\begin{align}\label{eq: D L2 bdd}
    \norm{\lambda \mathbb{D} \E_\lambda u}_2 \leq C \norm{u}_2.
\end{align}
\end{prop}

\section{Toolbox for \texorpdfstring{ $\L^p-\L^q$}{} bounded families} \label{sec: Toolbox}

We recall, in the parabolic scaling, several general abstract principles concerning  $\L^{p}-\L^q$ boundedness for families $(T_\lambda)_{\lambda \in \mathcal{U}}$ of bounded operators acting between $\L^2(\R^{n+1})$-spaces. In our case, these families will mostly be powers of the resolvent family $(\E_\lambda)_{\lambda>0}$ or their parabolic gradients.

\begin{defn}\label{def: toolbox}
Let $\mathcal{U} \subseteq \C \setminus \{0\}$ and $(T_\lambda)_{\lambda \in \mathcal{U}}$ be an operator family as above. Let $1\le p \le q \le \infty$. Then $(T_\lambda)_{\lambda \in \mathcal{U}}$ is said to 
\begin{enumerate}[(i)]
    \item be  $\L^p-\L^q$ bounded if there exists a constant $C$ such that
    \begin{equation*}
        \| T_\lambda f\|_q \le C |\lambda|^{-\gamma_{p,q}} \|  f\|_p,
    \end{equation*}
    for all $\lambda\in \mathcal{U}$ and all $f \in \L^p(\R^{n+1}) \cap \L^2(\R^{n+1})$.
    \item satisfy the  $\L^p-\L^q$ off-diagonal estimates if there exist constants $C,c$ such that
    \begin{equation*}
        \| \one_F T_\lambda (\one_Ef)\|_q \le C |\lambda|^{-\gamma_{p,q}}\e^{-c \frac{\d(E,F)}{|\lambda|}} \| \one_Ef\|_p,
    \end{equation*}
    for all $\lambda\in \mathcal{U}$, all $f \in \L^p(\R^{n+1}) \cap \L^2(\R^{n+1})$ and all measurable sets $E,F\subseteq\mathbb{R}^{n+1}$.
\end{enumerate}
When $p=q$, we speak of  $\L^p$ boundedness and  $\L^p$ off-diagonal estimates, respectively.
\end{defn}

The following is proved in \cite[Chapter 4]{AEbook2023}. (Proofs have nothing to do with the particular elliptic operators under consideration in this reference.)

\begin{lem}\label{lem: Toolbox}
Let $(T_\lambda)_{\lambda \in \mathcal{U}}$ and $(S_\lambda)_{\lambda \in \mathcal{U}}$ be operator families as above. Let $1\le p \le q \le r \le \infty$.
\begin{enumerate}[(i)]
    \item \label{item: Duality principle} \emph{(Duality)} $(T_\lambda)_{\lambda \in \mathcal{U}}$ is  $\L^p-\L^q$ bounded if and only if $(T^\star_\lambda)_{\lambda \in \mathcal{U}}$ is  $\L^{q'}-\L^{p'}$ bounded.
    \item \emph{(Composition)}\label{item: Composition} If $(T_\lambda)_{\lambda \in \mathcal{U}}$ is  $\L^p-\L^q$ bounded and $(S_\lambda)_{\lambda \in \mathcal{U}}$ is  $\L^q-\L^r$ bounded, then $(S_\lambda T_\lambda)_{\lambda \in \mathcal{U}}$ is  $\L^p-\L^r$ bounded.
    \item \label{item: Interpolation principle 1} \emph{(Interpolation 1)} If $(T_\lambda)_{\lambda \in \mathcal{U}}$ is  $\L^{p_0}-\L^{q_0}$ bounded and $\L^{p_1}-\L^{q_1}$ bounded, then $(T_\lambda)_{\lambda \in \mathcal{U}}$ is  $\L^{[p_0,p_1]_\theta}-\L^{[q_0,q_1]_\theta}$ bounded for every $\theta \in [0,1]$. 
    \item \label{item: Interpolation principle} \emph{(Interpolation 2)} If $(T_\lambda)_{\lambda \in \mathcal{U}}$ is  $\L^{p_0}-\L^{q_0}$ bounded and satisfies $\L^{p_1}-\L^{q_1}$ off-diagonal estimates, then $(T_\lambda)_{\lambda \in \mathcal{U}}$ satisfies $\L^{[p_0,p_1]_\theta}-\L^{[q_0,q_1]_\theta}$ off-diagonal estimates for every $\theta \in (0,1)$. 
    \item \emph{(Extrapolation)}\label{item: Extrapolation} If $(T_\lambda)_{\lambda \in \mathcal{U}}$ satisfies  $\L^p-\L^q$ off-diagonal estimates, then it is  $\L^p$ bounded and $\L^q$ bounded.
\end{enumerate}
Parts \ref{item: Duality principle} -- \ref{item: Interpolation principle 1} remain valid if boundedness is replaced by off-diagonal estimates.
\end{lem}

We also recall a bootstrapping argument which follows \emph{verbatim} from the proof of \cite[Lemma~4.4]{AEbook2023}, where the case $q=2$ was discussed. 

\begin{lem}[Triangle interpolation]\label{lem: gives m}
Let $q\in (1,\infty]$ and $(T_\lambda)_{\lambda \in \mathcal{U}}$ be $\L^q$ bounded. Assume that there exist $p,\varrho \in [1,q)$ such that $(T_\lambda)_{\lambda \in \mathcal{U}}$ is  $\L^p$ bounded and  $\L^\varrho-\L^q$ bounded. Then, for all $r\in (p,q]$, there exists an integer $m \ge 1$ such that $(T^m_\lambda)_{\lambda \in \mathcal{U}}$ is  $\L^r-\L^q$ bounded. 
\end{lem}

\section{Space-time decay for \texorpdfstring{$(\lambda \D \mathcal{E}_\lambda)_{\lambda>0}$}{}}\label{S4}

In this section, we prove off-diagonal decay for $(\lambda \mathbb{D} \mathcal{E}_\lambda)_{\lambda>0}$. Our main contribution concerns $(\lambda D_t^{1/2} \mathcal{E}_\lambda)_{\lambda>0}$ involving the non-local operator $D_t^{1/2}$. Our strategy is to split the supports into spatial and temporal components and estimate them separately. 

\subsection{Off-diagonal estimates for \texorpdfstring{$\boldsymbol{\lambda D_t^{1/2} \mathcal{E}_\lambda}$}{}}

We begin with the spatial supports, where the argument is very similar to the proof of \cite[Proposition~2.7]{AEN2025}.
		
\begin{prop}[Off-diagonal estimates on spatial supports]\label{prop: L2 spatial supports}
There exist constants $C$ and $c$ such that the following off-diagonal estimates hold: 
\begin{equation}\label{eq: Dt space off}
\|\one_{F\times \R} \, \lambda D_t^{1/2} \mathcal{E}_\lambda (\one_{E\times\R} \, u) \|_{2} \le C \e^{-c\frac{\d(E,F)}{\lambda}} \|\one_{E\times \R} \, u\|_{2}, 
\end{equation}
for all $\lambda>0$, measurable sets $E,F\subseteq\mathbb{R}^n$ and $u\in \L^2(\R^{n+1})$, where $\d(E,F)$ denotes the Euclidean distance between $E$ and $F$ in $\mathbb{R}^n$.
\end{prop}
\begin{proof}
It suffices to prove that
\begin{equation*}
     \norm{\one_{F\times\R} (\lambda \mathcal{E}_\lambda D_t^{1/2}(\one_{E\times\R} u)  )}_2 
    \le C \e^{-c \frac{\d(E,F)}{\lambda}} \norm{\one_{E\times\R} u }_2.
\end{equation*}
Indeed, the same inequality also holds for $(\mathcal{E}^\star_\lambda)_{\lambda>0}$ (that is, for $\H^\star$ in place of $\H$), and we conclude by duality.

Let $\d\coloneqq \d(E,F)$. It suffices to treat the case $\lambda \le \alpha \d$, where $\alpha \in (0,1)$ is a free parameter to be fixed later. Indeed, if $\lambda > \alpha \d$, then $1\leq\e^{\nicefrac{1}{\alpha}}\e^{-\nicefrac{\d}{\lambda}}$, and the claim follows already from \eqref{eq: D L2 bdd}. We fix $\widetilde{\eta} \in \Cont_\mathrm{b}^\infty(\mathbb{R}^n)$ such that $\widetilde{\eta} = 0$ on $E$, $\widetilde{\eta} = 1$ on $F$, $0 \le \widetilde{\eta} \le 1$ and $\norm{\nabla_x \widetilde{\eta}}_{L^\infty(\mathbb{R}^n)} \le \nicefrac{c}{d}$ with $c>0$ depending only on $n$. We then set
\begin{equation*}
\eta \coloneqq\e^{\frac{\alpha \d}{\lambda} \widetilde{\eta}} - 1, \qquad
w \coloneqq \lambda \mathcal{E}_\lambda D_t^{1/2}(\one_{E\times\R} u), \qquad
\text{and} \quad v \coloneqq  w \eta^2.
\end{equation*}
We have $(1+\lambda^2 \H) w = \lambda D_t^{1/2}(\one_{E\times\R} u)$ in $\mathrm{E}^\star$, and testing against $v$ gives
\begin{align*}
        \iint_{\R^{n+1}} \abs{w}^2 \eta^2 &+ \lambda^2 H_t D_t^{1/2} w \cdot \overline{D_t^{1/2} (w \eta^2)}+ \lambda^2  A \nabla_x w \cdot \overline{\nabla_x(w \eta^2)}  \\
        &= \lambda \iint_{\R^{n+1}} \one_{E\times\R} u \cdot \overline{D_t^{1/2}(w\eta^2)}. 
\end{align*}
As $D_t^{1/2}(w\eta^2)=\eta^2 D_t^{1/2}w$ by $t$-independence of $\eta$, we see from $\eta^2 \one_{E\times\R} u=0$ that the right-hand side vanishes and from skew-adjointness of $H_t$ on $\L^2(\R^{n+1})$ that so does the real part of the second term on the left:
\begin{equation*}
   \mathrm{Re}\left ( \iint_{\R^{n+1}} H_t D_t^{1/2} w \cdot \overline{D_t^{1/2} (w \eta^2)}  \right ) =\mathrm{Re}\left ( \iint_{\R^{n+1}} H_t (\eta D_t^{1/2} w) \cdot \overline{\eta D_t^{1/2}w} \right ) = 0.
\end{equation*}
Thus, taking the real part in the variational equality above, we obtain
\begin{equation*}
\iint_{\R^{n+1}} \abs{w}^2 \eta^2  +  \lambda^2 \iint_{\R^{n+1}} \mathrm{Re}(A \, \eta \nabla_x w \cdot \overline{\eta\nabla_xw})    
        = -2\lambda^2 \iint_{\R^{n+1}} \mathrm{Re}(A \nabla_x w \cdot \overline{\nabla_x\eta}) \, \eta  w .
\end{equation*}
Boundedness and ellipticity of $A$ as in \eqref{eq: ellipticité A} yield
\begin{align*}
\iint_{\R^{n+1}} \abs{w}^2 \eta^2  +  \lambda^2 \nu \iint_{\R^{n+1}} | \eta \nabla_x w |^2     
        &\le 2M\lambda^2 \iint_{\R^{n+1}} |\eta \nabla_x w | |w \nabla_x \eta|
        \\&\le \frac{M\lambda^2}{\varepsilon} \iint_{\R^{n+1}}  |w \nabla_x \eta|^2 +M\lambda^2 \varepsilon \iint_{\R^{n+1}} |\eta \nabla_x w |^2,
\end{align*}
where we have used the inequality $2ab \le \frac{a^2}{\varepsilon} + \varepsilon b^2$, with $\varepsilon>0$ to be fixed later. Thus,  
\begin{equation*}
    \iint_{\R^{n+1}} \abs{w}^2 \eta^2  + (\nu-M\varepsilon) \lambda^2  \iint_{\R^{n+1}} |\nabla_x w |^2 \eta^2 \le \frac{M\lambda^2}{\varepsilon} \iint_{\R^{n+1}}  |w|^2 | \nabla_x \eta|^2. 
\end{equation*}
Since $\nabla_x \eta = \frac{\alpha \d}{\lambda} (\eta + 1) \nabla_x \widetilde{\eta}$, we have $\lvert \nabla_x \eta \rvert \le \frac{c \alpha}{\lambda} (\eta + 1)$, and using the elementary inequality $(a+b)^2 \le 2a^2 + 2b^2$ in order to absorb terms from the right to the left, we obtain
\begin{align*}
    \left ( 1-\frac{2c^2M\alpha^2}{\varepsilon} \right )\iint_{\R^{n+1}} \abs{w}^2 \eta^2  + (\nu-M\varepsilon) \lambda^2  \iint_{\R^{n+1}} |\nabla_x w |^2 \eta^2 
    \le \frac{2c^2M\alpha^2}{\varepsilon} \iint_{\R^{n+1}}  |w|^2.
\end{align*}
At this point, we pick $\varepsilon \coloneqq \nicefrac{\nu}{2M}$, so that $\nu - M \varepsilon = \nicefrac{\nu}{2}$, and choose $\alpha$ sufficiently small so that $1 - \nicefrac{2c^2 M \alpha^2}{\varepsilon} \geq \nicefrac{1}{2}$. Since we have $\eta \ge\e^{\nicefrac{\alpha \d}{\lambda}} - 1 \ge \nicefrac{1}{2}\e^{\nicefrac{\alpha \d}{\lambda}}$ on $F$, we have shown that
 \begin{equation*}
    \e^{\alpha \frac{\d}{\lambda}} \Bigl(\norm{\one_{F\times\R} w}_2+\norm{\one_{F\times\R} \lambda \nabla_x  w }_2 \Bigr) \le C  \norm{w}_2.
\end{equation*}
Since $\|w\|_2 \leq C \|\one_{E\times\R} u\|_2$ by \eqref{eq: D L2 bdd}, the bound for the first term on the left is the required off-diagonal estimate.
\end{proof}
\begin{rem}
In fact, keeping the second term on the left-hand side in the final equation shows that for all $\mathbf{u} \in \L^2(\mathbb{R}^{n+1})^n$ we also have
    \begin{equation*}
    \|\one_{F\times \R} \, \lambda D_t^{1/2}  \mathcal{E}_\lambda \big( \lambda \, \Div_x  (\one_{E \times \R } \, \textbf{u}) \big) \|_{2} \le C \e^{-c\frac{\d(E,F)}{\lambda}} \|\one_{E\times \R} \, \textbf{u}\|_{2},
\end{equation*}
where $C$ and $c$ are the same constants as above.
\end{rem}

We now turn to the temporal supports, for which we obtain the following key result. On a first reading, we suggest to simply think of $p=2$. It will become clear later on why, in contrast to Proposition~\ref{prop: classical ODEs} we need a $p$-adapted bound.

\begin{prop}[Off-diagonal estimates on temporal supports]\label{prop: time supports}
Let $p \in (1,2]$ such that the families $(\mathcal{E}_\lambda)_{\lambda>0}$ and $(\lambda D_t^{1/2}\mathcal{E}_\lambda)_{\lambda>0}$ are $\L^p$ bounded. For every $N>1$, there exists a constant $C$ such that the off-diagonal estimate
\begin{equation*}
    \|\one_{\R^n\times F} \, \lambda D_t^{1/2}\mathcal{E}_\lambda  \left (  \one_{\R^n\times E} u \right ) \|_{p}
        \leq C\biggl(\frac{\lambda}{r} + \left(\frac{\lambda}{r}\right)^2\biggr)
        N^{-j\varepsilon}
        \|\one_{\R^n\times E} u\|_{p}
\end{equation*}
holds for $\lambda, r > 0$, $j \ge 1$ and $u \in \L^p(\R^{n+1})\cap \L^2(\R^{n+1})$ in the following two scenarios:
\begin{enumerate}[(i)]
    \item\label{item: cube to annulus time} \emph{(From the cube to the annulus)} For $E= I_r$, $F= N^{j+1}I_{r} \setminus N^{j} I_{r}$ with $\varepsilon = 1+\frac{1}{1+p'}$.
    
    \item\label{item: annulus to cube time} \emph{(From the annulus to the cube)} 
    For $E= N^{j+1}I_r \setminus N^j I_r$, $F= I_r$ with $\varepsilon = 2$.
\end{enumerate}
\end{prop}
\begin{proof}
    In both scenarios we shall prove the dual estimate in $\L^{p'}$. We write $[A,B] = AB - BA$ for the commutator of suitable operators. Let $u \in \EE$.
    Our main observation is that for $\eta \in \Cont^\infty_\mathrm{b}(\R)$, which we identify with a multiplication operator in the $t$-variable only and $u\in \mathrm{E}$, we have that
    \begin{align*}
        [\H^\star, \eta]u = [-\partial_t - \Div_x(A^\star \nabla_x), \eta]u = - (\partial_t \eta) u
    \end{align*}
    is a `local' operator. This identity is easily checked from the variational definition of $\H^\star$ if $u \in \Cont_\c^\infty(\R^{n+1})$ and extends to $\EE$ by density. Thus, for $\lambda >0$ we can use the commutator identity $[A,B]=B[B^{-1},A]B$ to find that
    \begin{align*}
        \eta \lambda \E^\star_\lambda D_t^{1/2}u&=\lambda [\eta, \E_\lambda^\star]D_t^{1/2}u+\lambda \E^\star_\lambda \eta D_t^{1/2}u
        \\&= \lambda \E^\star_\lambda [1+\lambda^2 \H^\star, \eta]\E_\lambda^\star D_t^{1/2}u+\lambda \E^\star_\lambda \eta D_t^{1/2}u
        \\&= -\lambda^3 \E^\star_\lambda (\partial_t \eta) \E^\star_\lambda D_t^{1/2}u+\lambda \E^\star_\lambda \eta D_t^{1/2}u
        \\&\eqqcolon \mathrm{I} + \mathrm{II}.
    \end{align*}
    By assumption and duality, the families $(\E^\star_\lambda)_{\lambda>0}$ and $(\lambda \E^\star_\lambda D_t^{1/2})_{\lambda>0}$ are $\L^{p'}$ bounded. Now, we let $E,F \subseteq \R$ be as in either of the two cases \ref{item: cube to annulus time}, \ref{item: annulus to cube time} and we let $u \in \Cont_\c^\infty(\R^{n+1})$ have its support in $\R^n \times F$. Note that such $u$ are dense in $\L^{p'}(\R^n \times F)$ by the specific form of $F$.
    
    We take $\eta \in \Cont_\mathrm{b}^\infty(\R)$ with $0\leq \eta\leq1$, $\eta=1$ on $E$, $\eta=0$ on $F$ and abbreviate $\widetilde{E}\coloneqq \supp(\eta)$. By the previous computations, we have
    \begin{align}\label{eq: I+II}
        \norm{\eta \lambda \E^\star_\lambda D_t^{1/2} u}_{p'} \le \norm{\mathrm{I}}_{p'}+\norm{\mathrm{II}}_{p'}.
    \end{align}
    Estimating $\norm{\mathrm{I}}_{p'}$ is straightforward: 
    \begin{align}\label{eq: I}
        \norm{\mathrm{I}}_{p'} \leq C  \lambda^2 \norm{\partial_t \eta}_{\infty} \norm{u}_{p'},
    \end{align}
    where $C>0$ depends on the $\L^p$ bound of $\E_\lambda$.
    
    For $\norm{\mathrm{II}}_{p'}$, we use the representation formula \eqref{eq: Dt representation} for $D^{1/2}_t u$ and that $u$ and $\eta$ have disjoint supports in order to obtain 
    \begin{align*}
        \abs{\eta D_t^{1/2}u}(x,t) \leq \abs{\eta(t)} \int_\R \frac{\abs{u(x,t) - u(x,s)}}{\abs{t-s}^{\nicefrac{3}{2}}} \dd s \leq \abs{\eta(t)} \int_\R \frac{\abs{u(x,s)}}{\abs{\d(\widetilde{E},F)}^{\nicefrac{3}{2}}} \dd s.
    \end{align*}
    Thus,
    \begin{align*}
        \norm{\mathrm{II}}_{p'} 
        \leq C \lambda \|\eta D_t^{1/2}u\|_{p'} 
        \leq \lambda \Biggl(\int_{\R^n} \int_\R \abs{\eta(t)}^{p'} \biggl(\int_\R \frac{\abs{u(x,s)}}{\abs{\d(\widetilde{E},F)}^{\nicefrac{3}{2}}} \dd s\biggr)^{p'} \dd t \dd x \Biggr)^{\nicefrac{1}{p'}}
    \end{align*}
    and Jensen's inequality gives
    \begin{align}
    \label{eq: time ODEs support dependency}
    \begin{split}
    \norm{\mathrm{II}}_{p'} 
        &\leq \frac{\lambda}{\abs{\d(\widetilde{E},F)}^{\nicefrac{3}{2}}} \Bigl(\int_\R \abs{\eta(t)}^{p'} \dd t\Bigr)^{\nicefrac{1}{p'}}  \Bigl(\int_{\R^n} \abs{F}^{p'-1}\int_\R \abs{u(x,s)}^{p'} \dd s \dd x\Bigr)^{\nicefrac{1}{p'}}\\
        &\leq \frac{\lambda \abs{\widetilde{E}}^{\nicefrac{1}{p'}} \abs{F}^{1-\nicefrac{1}{p'}}}{\d(\widetilde{E},F)^{\nicefrac{3}{2}}} \norm{u}_{p'}.
    \end{split}
    \end{align}
    All we have to do now is to plug in appropriate estimates of the occurring quantities.
    
    \noindent\emph{\ref{item: cube to annulus time} From the cube to the annulus}
    
    In this case we have $F = N^{j+1}I_r\setminus N^j I_r$ and we pick $\eta$ such that $\supp(\eta) = \widetilde{E} = N^{\gamma j} I_r\supseteq I_r = E$ for some $\gamma\in(0,1)$ that is yet to be determined. We have
    \begin{align*}
        \abs{F} \leq C N^{2j}r^2, \quad \abs{\widetilde{E}} \leq C N^{2\gamma j}r^2, \quad \d(\widetilde{E},F) \geq C N^{2j}r^2,
    \end{align*}
    and we can arrange that
    \begin{align*}
        \norm{\partial_t \eta}_\infty \leq C (\d(\widetilde{E}^\mathrm{c}, E))^{-1} \leq C N^{-2\gamma j}r^{-2},
    \end{align*}
    where $C$ now depends on $N$ and $\gamma$ as well. Plugged into \eqref{eq: I+II} - \eqref{eq: time ODEs support dependency}, this gives
    \begin{align*}
        \|\one_{\R^n\times E} \, \lambda D_t^{1/2}\mathcal{E}_\lambda^\star  u \|_{p'} &\leq C \Bigl(\frac{\lambda}{r} N^{-j(1 + \frac{2}{p'}-\frac{2\gamma}{p'})} + \frac{\lambda^2}{r^2} N^{-2\gamma j}\Bigr) \norm{u}_{p'}.
    \end{align*}
    Optimizing the estimate by equalizing the exponents of $N$ yields the optimal choice 
    \begin{align}\label{eq: epsilon to optimize}
        \gamma = \frac{3p-2}{4p-2} = \frac{1}{2}\Bigl(1+\frac{1}{1+p'}\Bigr)
    \end{align}
    and this leads to the desired decay estimate
    \begin{align*}
        \|\one_{\R^n\times E} \, \lambda D_t^{1/2}\mathcal{E}_\lambda^\star   u \|_{p'}  &\leq C \Bigl(\frac{\lambda}{r} + \frac{\lambda^2}{r^2}\Bigr) N^{-j\bigl(1+\frac{1}{1+p'}\bigr)} \norm{u}_{p'}.
    \end{align*}

    \noindent\emph{\ref{item: annulus to cube time} From the annulus to the cube}
    
    In this case we have $F = I_r$ and we pick $\eta$ such that $\supp(\eta) = \widetilde{E} = N^{j+2}I_r \setminus N^{j-\nicefrac{1}{2}} I_r \supseteq N^{j+1}I_r \setminus N^{j} I_r = E$. We have
    \begin{align*}
        \abs{F} \leq C r^2,\quad \abs{\widetilde{E}} \leq C N^{2 j}r^2,\quad \d(\widetilde{E},F) \geq C N^{2j}r^2
    \end{align*}
    and we can arrange that
    \begin{align*}
        \norm{\partial_t \eta}_\infty &\leq C(\d(\widetilde{E}^\mathrm{c}, E))^{-1} \leq C N^{-2 j}r^{-2}.
    \end{align*}
    This gives the desired decay estimate
    \begin{align*}
        \|\one_{\R^n\times E} \, \lambda D_t^{1/2}\mathcal{E}_\lambda^\star   u  \|_{p'} &\leq C \Bigl(\frac{\lambda}{r} N^{-j(3 - \frac{2}{p'})} + \frac{\lambda^2}{r^2} N^{-2 j}\Bigr) \norm{u}_{p'}\\
        &\leq C \Bigl(\frac{\lambda}{r} + \frac{\lambda^2}{r^2}\Bigr) N^{-2j} \norm{u}_{p'}. \qedhere
    \end{align*}
\end{proof}

\begin{rem}\label{rem: epsilon}
    Proposition \ref{prop: time supports} \ref{item: cube to annulus time} admits some flexibility when going from the cube to the annulus. We can also allow for $F=N^{j+1} I_r\setminus N^j I_r$ and a stretched interval $E=N^{\nicefrac{j}{2}} I_r$, that is
    \begin{align*}
        \|\one_{\R^n \times (N^{j+1}I_r \setminus N^j I_r)}\, \lambda D_t^{1/2}\mathcal{E}_\lambda  ( \one_{\R^n \times N^{\nicefrac{j}{2}}I_r}  u) \|_{p} \leq C\left(\frac{\lambda}{r} + \biggl(\frac{\lambda}{r}\right)^2\biggr)N^{-j\bigl(1+\frac{1}{1+p'}\bigr)} \|\one_{\R^n \times N^{\nicefrac{j}{2}}I_r} u\|_{p}
    \end{align*}
    holds for all $u \in \L^p(\mathbb{R}^{n+1})\cap \L^2(\mathbb{R}^{n+1})$.\\
    Indeed, in the proof of Proposition \ref{prop: time supports} \ref{item: cube to annulus time}, we see that $\supp(\eta)\supseteq E = N^{\nicefrac{j}{2}} I_r$ requires $\gamma>\nicefrac{1}{2}$, 
    which is the case for the choice made in \eqref{eq: epsilon to optimize}; the remainder of the proof yields the same four estimates for $\abs{F},\abs{\widetilde{E}}, \d(\widetilde{E},F)$ and $\norm{\partial_t \eta}_\infty$.
\end{rem}

\subsection{Full gradient off-diagonal estimates and composition} 

We are now ready to assemble full space-time off-diagonal bounds. 

\begin{thm}\label{thm: parabolicODEs}
Let $\varrho \in [1,2]$ such that the families $(\mathcal{E}_\lambda)_{\lambda>0}$ and $(\lambda \mathbb{D}\mathcal{E}_\lambda)_{\lambda>0}$ are $\L^\varrho$ bounded. Let $p\in(\varrho,2]$ or $p=\varrho=2$. Fix $N\ge 4$ and an integer $m \ge 1$. Then there exists a constant $C$ such that the off-diagonal estimate
\begin{align*}
    \|\one_F \lambda \D\mathcal{E}^m_\lambda \one_E u \|_{p}
        \leq C\biggl(\frac{\lambda}{r} +\Bigl(\frac{\lambda}{r}\Bigr)^{4N}\biggr)
        N^{-j \varepsilon}
        \|\one_E u\|_{p}
\end{align*}
holds for $\lambda, r >0$, $j \geq 2$ and $u \in \L^p(\R^{n+1}) \cap \L^2(\R^{n+1})$ in the following two scenarios:
\begin{enumerate}[(i)]
    \item\label{item: cube to annulus} \emph{(From the cube to the annulus)} For $E= \Delta_r$, $F= C_j^N(\Delta_r)$ with $\varepsilon = 1+\frac{1}{1+p'}$.
    
    \item\label{item: annulus to cube} \emph{(From the annulus to the cube)} 
    For $E= C_j^N(\Delta_r)$, $F= \Delta_r$ with $\varepsilon = 2$.
\end{enumerate}
\end{thm}

\begin{proof}
The family $(\lambda \nabla_x \mathcal{E}^m_\lambda)_{\lambda>0}$ satisfies $\L^2$ off-diagonal estimates by composition (Proposition~\ref{prop: classical ODEs} and Lemma \ref{lem: Toolbox} \ref{item: Composition}). This exponential decay is strong enough to conclude \ref{item: cube to annulus} and \ref{item: annulus to cube} for $(\lambda \nabla_x \mathcal{E}^m_\lambda)_{\lambda>0}$ simply by interpolation and we can take $\varepsilon = 2$ in either case. Indeed, by interpolation with $\L^\varrho$ boundedness, $(\lambda \nabla_x \mathcal{E}^m_\lambda)_{\lambda>0}$ satisfies $\L^p$ off-diagonal estimates. Since for $j\ge 2$ and $N \geq 2$ we have
\begin{equation}\label{eq: distance p}
    \d(C^N_j(\Delta_r), \Delta_r) \ge \min\left(2^jr-r,\sqrt{(N^jr)^2-r^2}\right) \ge 2^{j-1} r,
\end{equation}
these off-diagonal bounds mean that
\begin{align*}
    \|\one_F \lambda \nabla_x \mathcal{E}^m_\lambda \one_E u \|_{p}
        \leq C \e^{-c \frac{2^j r}{\lambda}}
        \|\one_E u\|_{p}
\end{align*}
with certain constants $C$, $c$. However, for some constant $C$ depending on $N$ and $c$ we have
\begin{equation}\label{eq: exp vs N}
    \e^{-c \frac{2^j r}{\lambda}} 
    \le C \left( \frac{2^j r}{\lambda} \right)^{-4N} 
    = C \left( \frac{\lambda}{r} \right)^{4N} (2^{2N})^{-2j}
    \le C \left( \frac{\lambda}{r} \right)^{4N} N^{-2j}.
\end{equation}

We are then left to prove \ref{item: cube to annulus} and \ref{item: annulus to cube} for $(\lambda D_t^{1/2} \mathcal{E}^m_\lambda)_{\lambda>0}$. In what follows, let $u\in\L^p(\R^{n+1}) \cap \L^2(\R^{n+1})$.

\noindent\textit{\ref{item: cube to annulus} From the cube to the annulus}

We write 
\begin{align*}
    C^N_j(\Delta_r) = C^{N,t}_j(\Delta_r) \cup C^{N,x}_j(\Delta_r),
\end{align*}
with
\begin{align*}
    C^{N,t}_j(\Delta_r) \coloneqq 2^{j+1} Q_r \times (N^{j+1} I_r \setminus N^j I_r) \quad \text{and} \quad 
    C^{N,x}_j(\Delta_r) \coloneqq (2^{j+1} Q_r \setminus 2^j Q_r) \times N^j I_r,
\end{align*}
as indicated in Figure~\ref{fig: time and space supports}.
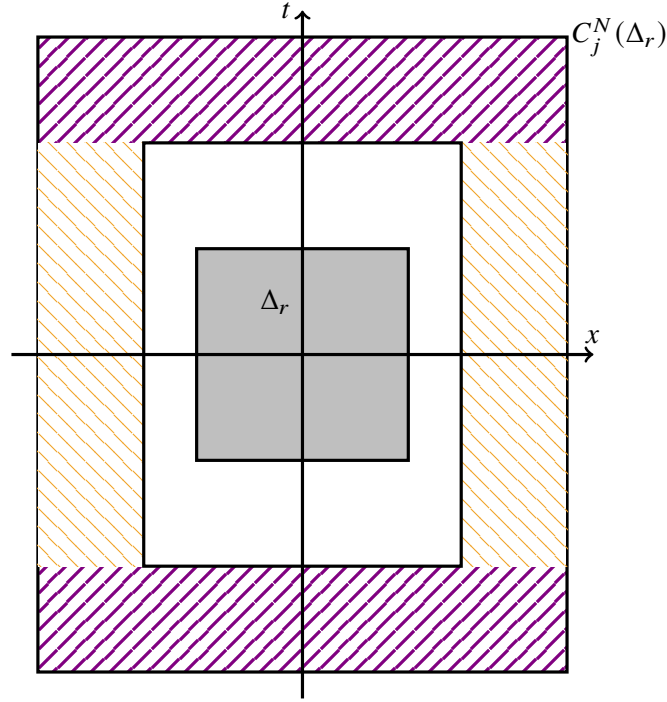
\begin{figure}[ht]
    \definecolor{gold}{HTML}{F1A226}
    \colorlet{color1}{violet}
    \colorlet{color2}{gold}
    \begin{tikzpicture}[scale=0.7]
        \pgfdeclarelayer{layer1}
        \pgfdeclarelayer{layer2}
        \pgfdeclarelayer{layer3}
        \pgfdeclarelayer{layer4}
        \pgfdeclarelayer{layer5}
        \pgfsetlayers{layer1,layer2,layer3,layer4,layer5}
        
        \tikzmath{\Ballx=2; \Ballt=2; \InnerAnnulusx=3; \InnerAnnulust=4; \OuterAnnulusx=5; \OuterAnnulust=6;}
        
        \begin{pgfonlayer}{layer1}
            \filldraw[very thick, pattern={Lines[angle=45,distance=5pt]}, pattern color=color1]
            (-\OuterAnnulusx,-\OuterAnnulust)
            rectangle (\OuterAnnulusx,\OuterAnnulust);
        \end{pgfonlayer}
        
        \begin{pgfonlayer}{layer2}
            \fill[fill=white]
            ({-\OuterAnnulusx+0.02},-\InnerAnnulust)
            rectangle (\OuterAnnulusx-0.02,\InnerAnnulust);
        \end{pgfonlayer}
        
        \begin{pgfonlayer}{layer3}
            \fill[pattern={Lines[angle=-45,distance=5pt]}, pattern color=color2]
            (-\OuterAnnulusx,-\InnerAnnulust)
            rectangle (\OuterAnnulusx,\InnerAnnulust);
        \end{pgfonlayer}
        
        \begin{pgfonlayer}{layer4}
            \filldraw[very thick, fill=white]
            (-\InnerAnnulusx,-\InnerAnnulust)
            rectangle (\InnerAnnulusx,\InnerAnnulust);
        \end{pgfonlayer}
        
        \begin{pgfonlayer}{layer5}
            \filldraw[very thick, fill=lightgray]
            (-\Ballx,-\Ballt)
            rectangle (\Ballx,\Ballt);

            \draw[black, very thick, ->] ({-\OuterAnnulusx-0.5},0) -- ({\OuterAnnulusx+0.5},0) node[above] {$x$};
            \draw[black, very thick, ->] (0,{-\OuterAnnulust-0.5}) -- (0,{\OuterAnnulust+0.5}) node[left] {$t$};
            
            \path ({\OuterAnnulusx+1},\OuterAnnulust) node {$C_j^N(\Delta_r)$}
            (-0.5,1) node {$\Delta_r$};
        \end{pgfonlayer}
    \end{tikzpicture}
    \caption{The splitting of $C_j^N(\Delta_r)$ in the proof of Theorem~\ref{thm: parabolicODEs} with the spatial support $C_j^{N,x}(\Delta_r)$ (\textcolor{color2}{golden}) and the time support $C_j^{N,t}(\Delta_r)$ (\textcolor{color1}{violet})}
    \label{fig: time and space supports}
\end{figure}
In the spirit of reproducing the proof of composition of off-diagonal estimates, we introduce the auxiliary set
\begin{align*}
    \Delta^{j-1,\nicefrac{j}{2}}_r \coloneqq 2^{j-1}Q_r \times N^{\nicefrac{j}{2}}I_r.
\end{align*}
Similarly to \eqref{eq: distance p} we check that for $N\geq4$ we have
\begin{align}\label{eq: distance auxiliary set}
    \d(C_j^{N,x}(\Delta_r),\Delta^{j-1,\nicefrac{j}{2}}_r) \geq 2^{j-1} r\quad\text{and}\quad \d((\Delta^{j-1,\nicefrac{j}{2}}_r)^\c, \Delta_r) \geq 2^{j-3}r.
\end{align}
We now split accordingly
\begin{align*}
    \one_{C_j^N(\Delta_r)} \lambda D_t^{1/2} \E_\lambda^m \one_{\Delta_r} u 
    &= \Bigl( \one_{C_j^{N,x}(\Delta_r)} \lambda D_t^{1/2} \E_\lambda \one_{\Delta^{j-1,\nicefrac{j}{2}}_r} \Bigr) \E_\lambda^{m-1} \one_{\Delta_r} u \\
    &\qquad + \Bigl(\one_{C_j^{N,t}(\Delta_r)} \lambda D_t^{1/2} \E_\lambda \one_{\Delta^{j-1,\nicefrac{j}{2}}_r} \Bigr) \E_\lambda^{m-1} \one_{\Delta_r} u \\
    &\qquad + \one_{C_j^N(\Delta_r)} \lambda D_t^{1/2} \E_\lambda \Bigl( \one_{(\Delta^{j-1,\nicefrac{j}{2}}_r)^\mathrm{c}} \E_\lambda^{m-1} \one_{\Delta_r} u \Bigr)\\
    &\eqqcolon \mathrm{I} + \mathrm{II} + \mathrm{III}.
\end{align*}
The terms in brackets will admit decay in $\L^p$ norm, while the ones without will be estimated by uniform $\L^p$ boundedness. 

For  $\mathrm{I}$, we interpolate the exponential decay of $\lambda D_t^{1/2} \E_\lambda$ in the spatial direction on $\L^2$ (Proposition \ref{prop: L2 spatial supports}) with uniform boundedness on $\L^\varrho$ via the Riesz--Thorin theorem. This yields
\begin{align*}
    \norm{\mathrm{I}}_p \leq C  \e^{-c \frac{2^j r}{\lambda}}  \norm{\E_\lambda^{m-1} \one_{\Delta_r} u}_p \leq  C \left( \frac{\lambda}{r} \right)^{4N} N^{-2j} \norm{ \one_{\Delta_r} u}_p,
\end{align*}
where the second step is just \eqref{eq: exp vs N}.
Similarly, for $\mathrm{III}$ we interpolate the exponential decay for the resolvent family on $\L^2$ (Proposition \ref{prop: classical ODEs}) with  $\L^\varrho$ boundedness:
\begin{align*}
    \norm{\mathrm{III}}_p \leq C \norm{\one_{(\Delta^{j-1,\nicefrac{j}{2}}_r )^\c} \E_\lambda^{m-1} \one_{\Delta_r} u }_p \leq C \left( \frac{\lambda}{r} \right)^{4N} N^{-2j} \norm{ \one_{\Delta_r} u}_p.
\end{align*}
In both cases, $C$ now depends on the $\L^p$ bounds of $(\mathcal{E}_\lambda)_{\lambda>0}$ and $(\lambda \mathbb{D}\mathcal{E}_\lambda)_{\lambda>0}$, $m$ and $N$. Finally, for $\mathrm{II}$ we use the decay on temporal supports (Proposition \ref{prop: time supports} \ref{item: cube to annulus} or rather its more general version in Remark \ref{rem: epsilon}) and find
\begin{align*}
    \norm{\mathrm{II}}_p \leq C \biggl(\frac{\lambda}{r} +\left(\frac{\lambda}{r}\right)^{2}\biggr)N^{-j\bigl( 1+\frac{1}{1+p'}\bigr)} \norm{\E_\lambda^{m-1} \one_{\Delta_r} u }_p \leq C \left(\frac{\lambda}{r} +\bigg(\frac{\lambda}{r}\right)^{2}\bigg)N^{-j\bigl( 1+\frac{1}{1+p'} \bigr)} \norm{ \one_{\Delta_r} u }_p.
\end{align*}
The claim for this case follows by collecting the above estimates and using that $1+\frac{1}{1+p'} \le 2$ for $\mathrm{I}$ and $\mathrm{III}$ as well as $(\nicefrac{\lambda}{r})^2 \le \nicefrac{\lambda}{r} + (\nicefrac{\lambda}{r})^{4N}$ for $\mathrm{II}$.

\noindent\textit{\ref{item: annulus to cube time} From the annulus to the cube}

Our strategy is exactly the same up to introducing 
\begin{align*}
    \widetilde{C}^N_j(\Delta_r) \coloneqq \left( 2^{j+2}Q_r \times N^{j+2} I_r \right ) \setminus \left ( 2^{j-1}Q_r \times N^{j-1} I_r \right)
\end{align*}
with its corresponding splitting in the $x$- and $t$-direction
\begin{align*}
    \widetilde{C}^{N,t}_j(\Delta_r) \coloneqq 2^{j+2} Q_r \times (N^{j+2} I_r \setminus N^{j-1} I_r) \quad \text{and} \quad 
    \widetilde{C}^{N,x}_j(\Delta_r) \coloneqq (2^{j+2} Q_r \setminus 2^{j-1} Q_r) \times N^{j-1} I_r
\end{align*}
and then splitting
\begin{align*}
    \one_{\Delta_r} \lambda D_t^{1/2} \E_\lambda^m \one_{C_j^N(\Delta_r)} u
    &= \one_{\Delta_r} \lambda D_t^{1/2} \E_\lambda \Bigl( \one_{(\widetilde{C}^N_j(\Delta_r))^\mathrm{c}} \E_\lambda^{m-1} \one_{C_j^{N}(\Delta_r)} u \Bigr) \\
    &\qquad + \Bigl( \one_{\Delta_r} \lambda D_t^{1/2} \E_\lambda \one_{\widetilde{C}^{N,x}_j(\Delta_r)} \Bigl) \E_\lambda^{m-1} \one_{C_j^{N}(\Delta_r)} u \\
    &\qquad + \Bigl( \one_{\Delta_r} \lambda D_t^{1/2} \E_\lambda \one_{\widetilde{C}^{N,t}_j(\Delta_r)} \Bigr) \E_\lambda^{m-1} \one_{C_j^N(\Delta_r)} u \\
    &\eqqcolon \mathrm{I} + \mathrm{II} + \mathrm{III}.
\end{align*}
The terms $\mathrm{I}$ and $\mathrm{II}$ are treated as before: The interpolated version of the off-diagonal estimates of the resolvent family on $\L^2$ from Proposition \ref{prop: classical ODEs} takes care of $\mathrm{I}$, while the separation in $\mathrm{II}$ in spatial direction allows for the use of Proposition \ref{prop: L2 spatial supports}:
\begin{align*}
    \norm{\mathrm{I}}_p + \norm{\mathrm{II}}_p \leq C \left( \frac{\lambda}{r} \right)^{4N} N^{-2j} \norm{ \one_{C_j^N(\Delta_r)} u}_p.
\end{align*}
For $\mathrm{III}$, we observe that $\widetilde{C}^{N,t}_j(\Delta_r)$, compared to $C_j^{N,t}(\Delta_r)$, contains the three generations of temporal bands $j-1$, $j$, $j+1$. Using the decay of the gradient family in the time direction to the three bands (Proposition \ref{prop: time supports} \ref{item: annulus to cube time}), we obtain
\begin{align*}
    \norm{\mathrm{III}}_p \leq C \left(\frac{\lambda}{r} +\biggl(\frac{\lambda}{r}\right)^{2}\biggr)N^{-2j} \norm{\E_\lambda^{m-1} \one_{\Delta_r} u }_p \leq C \left(\frac{\lambda}{r} +\biggl(\frac{\lambda}{r}\right)^{2}\biggr)N^{-2j} \norm{ \one_{\Delta_r} u }_p.
\end{align*}
The claim again follows by collecting the above estimates and using that $(\nicefrac{\lambda}{r})^2 \le \nicefrac{\lambda}{r} + (\nicefrac{\lambda}{r})^{4N}$.
\end{proof}

\begin{rem}\label{rem: epsilon2}
We again have some flexibility when going from the cube to the annulus in Theorem~\ref{thm: parabolicODEs} \ref{item: cube to annulus}. In the notation of the last proof, we can replace $E = \Delta_r$ by $E = \Delta_r^{j-\nicefrac{3}{2},\nicefrac{j}{4}} = 2^{j-\nicefrac{3}{2}}Q_r \times N^{\nicefrac{j}{4}}I_r$ and obtain
\begin{align*}
    \|\one_{C^N_j(\Delta_r)}\, \lambda \mathbb{D}\mathcal{E}^m_\lambda  \bigl( \one_{\Delta^{j-\nicefrac{3}{2},\nicefrac{j}{4}}_r}  u \bigr) \|_{p} \leq\ C\biggl(\frac{\lambda}{r} +\biggl(\frac{\lambda}{r}\biggr)^{4N}\biggr)
    N^{-j\bigl(1+\frac{1}{1+p'}\bigr)}
    \|\one_{\Delta^{j-\nicefrac{3}{2},\nicefrac{j}{4}}_r} u\|_{p}.
\end{align*}
For this, we only need to check that \eqref{eq: distance auxiliary set} and \eqref{eq: distance p} still hold with the new choice of $E$. But a quick inspection again reveals that this is the case. 
\end{rem}

\begin{rem}
In \cite[Proposition 2.11]{AEN2020} the authors obtained a similar conclusion but with the decay parameter $\varepsilon > 0$ unspecified from Shneiberg’s stability theorem \cite{Shneiberg}; see also \cite[Appendix]{ABES2019}. The improvement here comes essentially from leveraging that $[\eta, \H^\star]$ in Proposition~\ref{prop: time supports} is a local multiplication operator.
\end{rem}

\section{Two boundedness criteria for operators with limited space-time decay}\label{S5}

Now that the nature of the decay in the off-diagonal estimates is known, we present two boundedness criteria from harmonic analysis adapted to this setting. On a first reading, we suggest that the reader skip the proofs and goes directly to the following sections where these results are applied. 

The first criterion relies on a weak type $(1,1)$ bound for the parabolic maximal function, whereas in contrast the second uses a domination by the iterated maximal function, which is not of weak type $(1,1)$.

\subsection{Extrapolation \`a la Blunck and Kunstmann}

The following is a two-scale version of Blunck and Kunstmann's criterion \cite{blunck2003calderon}; see also \cite{auscher2007necessary}. We write $\Z$ for the class of simple functions on $\R^{n+1}$ with support of finite measure.

\begin{thm}[Blunck--Kunstmann Extrapolation]\label{thm: BK}
	Let $1 \le p < q < \infty$ and $T$ be a sublinear operator of strong type $(q,q)$ on $\L^q(\R^{n+1})$. Let $N>1$ and $(\varepsilon_j)$ be a sequence in $(0,\infty)$ such that
	\begin{equation*}
		S\coloneqq \sum_{j=1}^{\infty} \varepsilon_j  ( 2^n N^2  )^{\nicefrac{j}{q'}}<\infty.
	\end{equation*}
	Assume that for each parabolic cube $\Delta$ and for all $u\in \Z$ supported in $\Delta$ there exists a decomposition $u=v+w$ with $v,w \in \L^q(\R^{n+1})$ satisfying
	\begin{align}
		\norm{\one_{C^N_j(\Delta)} T w}_q &\le \varepsilon_j r(\Delta)^{-\gamma_{p,q}} \| u \|_p	\qquad (j\geq 2) \label{eq: BK1}\\
		\intertext{and}
		\norm{\one_{C^N_j(\Delta)}v}_q &\le \varepsilon_j r(\Delta)^{-\gamma_{p,q}} \| u \|_p	\qquad (j\geq 1). \label{eq: BK2}
	\end{align}
	Then $T$ is of weak type $(p,p)$ on $\Z$ with a bound depending only on $S$, $N$, the dimension $n$ and the strong type $(q,q)$ bound of $T$. In particular, $T$ is of strong type $(s,s)$ on $\Z$ for all $s \in (p,q]$.
\end{thm}

\begin{proof}
	Let $\alpha>0$ and $u\in \Z$. We employ the usual dyadic Calderón--Zygmund decomposition at height $\alpha$: There is a collection $\mathcal{D}$ of parabolic cubes $\Delta$ such that we can decompose $u=g + \sum_{\Delta \in \mathcal{D}}b_\Delta$ with $g,b_\Delta\in \Z$ and
	\begin{enumerate}[(i)]
		\item $\displaystyle \norm{g}_\infty \leq C \alpha$ and $g=\one_{(\cup_{\Delta \in \mathcal{D}} \Delta)^\mathrm{c}} u$,\label{eq: CZ-i}
		\item $\displaystyle \norm{b_\Delta}_p\leq C \alpha\abs{\Delta}^{1/p}$ and $b_\Delta = \one_\Delta u$, \label{eq: CZ-ii}
		\item $\displaystyle \sum_{\Delta \in \mathcal{D}} \abs{\Delta} \leq C \alpha^{-p} \norm{u}_p^p$, \label{eq: CZ-iii}
	\end{enumerate}
	where $C$ depends on $p$ and $n$. More precisely, $\mathcal{D}$ is the collection of maximal half-open dyadic cubes in $\{\M(|u|^p) > \alpha^p\}$, where $\M$ is the dyadic parabolic maximal operator.

    For the rest of the proof, $C$ may vary from line to line and we emphasize new parameter-dependencies.
	
	We put $\Omega\coloneqq \cup_{\Delta\in\mathcal{D}}\,C_1^N(\Delta)$. It suffices to show that 
	\begin{align}\label{eq: sufficient for weak bound BK}
		\norm{Tu}_{\L^q(\R^{n+1}\setminus\Omega)}^q \leq C \alpha^{q-p}\norm{u}_p^p
	\end{align}
	as \ref{eq: CZ-iii} in conjunction with Chebyshev's inequality implies
	\begin{align*}
		\abs{\{(x,t)\in\R^{n+1}: \abs{Tu(x,t)} > \alpha\}}  &\leq \abs{\Omega} + \alpha^{-q} \norm{Tu}_{\L^q(\R^{n+1}\setminus\Omega)}^q \\
		&\leq 4^n N^4 \sum_{\Delta \in \mathcal{D}} \abs{\Delta} + C \alpha^{-p}\norm{u}_p^p \\
		&\leq C \alpha^{-p} \norm{u}_p^p,
	\end{align*}
	where $C$ also depends on $N$. From here, the claim follows by Marcinkiewicz interpolation. The rest of the proof comes in two steps.
    
	\emph{Step 1: Decomposing $u$.}
    
    For each $\Delta\in\mathcal{D}$ we have the decomposition $b_\Delta = v_\Delta + w_\Delta$ from the assumption. We decompose
	\begin{align}\label{eq: BK decomposition of u}
		u = g + \sum_{\Delta \in \mathcal{D}}v_\Delta  + \sum_{\Delta \in \mathcal{D}} w_\Delta
	\end{align}
	and need to make sure that each term is well-defined and belongs to $\L^q(\R^{n+1})$. Since $u \in \Z$, this is clear for $g$. We estimate $\sum_{\Delta \in \mathcal{D}}v_\Delta$ in $\L^q(\R^{n+1})$ by duality. This estimate will also reveal that the sum converges absolutely pointwise almost everywhere. Let $f \in \L^{q'}(\R^{n+1})$ with $\norm{f}_{q'} \leq 1$. We have
	\begin{align}\label{eq: Step 2 duality}
    \begin{split}
		\Bigl\vert\iint_{\R^{n+1}} \Bigl(\sum_{\Delta \in \mathcal{D}}|v_\Delta|\Bigr) f\Bigr\vert
		&\leq \sum_{\Delta \in \mathcal{D}} \iint_{\R^{n+1}} \abs{v_\Delta f} = \sum_{\Delta \in \mathcal{D}} \sum_{j\geq 1} \iint_{C_j^N(\Delta)} \abs{v_\Delta f}\\
		&\leq \sum_{\Delta \in \mathcal{D}} \sum_{j\geq 1} \norm{v_\Delta}_{\L^q(C_j^N(\Delta))} \norm{f}_{\L^{q'}(C_j^N(\Delta))}.
    \end{split}
	\end{align}
	Using \eqref{eq: BK2} and then \ref{eq: CZ-ii}, we obtain
	\begin{align}\label{eq: BK v estimate}
		\norm{v_\Delta}_{\L^q(C_j^N(\Delta))} \leq \varepsilon_j \abs{\Delta}^{\frac{1}{q}-\frac{1}{p}} \norm{b_\Delta}_p \leq C \alpha \varepsilon_j \abs{\Delta}^{\frac{1}{q}} .
	\end{align}
	We would like to estimate the norm of $f$ in terms of the maximal function on $\R^{n+1}$. However, there is no single metric on $\R^{n+1}$ for which all sets $2^jQ \times N^jI$ with $Q$, $I$, $j$ as above are metric balls. For now, we define $\M_j$ as the maximal function on $\R^{n+1}$ that is associated to the metric
	\begin{align*}
		\d^{(j)}((x,t),(y,s)) \coloneqq \max\biggl\{\frac{\abs{x-y}_\infty}{2^{j+1}}, \frac{\abs{t-s}^{\nicefrac{1}{2}}}{N^{j+1}}\biggr\},
	\end{align*}
	so that for fixed $j$, the sets $2^{j+1}Q \times N^{j+1}I$ are metric balls of radius $r(\Delta)$ with respect to $\d^{(j)}$. We can then estimate 
	\begin{align*}
		\norm{f}_{\L^{q'}(C_j^N(\Delta))}  &\leq (2^n N^2)^{\frac{j+1}{q'}}\abs{\Delta}^{\frac{1}{q'}} \biggl(\bariint_{2^{j+1}Q\times N^{j+1}I} \abs{f}^{q'} \biggr)^{\nicefrac{1}{q'}}\\
		&\leq (2^n N^2)^{\frac{j+1}{q'}}\abs{\Delta}^{\frac{1}{q'}} \bigl(\M_j(\abs{f}^{q'})(y)\bigr)^{\frac{1}{q'}}
	\end{align*}
	for every $y \in \Delta$. Taking averages in $y$ yields
	\begin{align}\label{eq: BK f estimate}
		\norm{f}_{\L^{q'}(C_j^N(\Delta))} \leq (2^n N^2)^{\frac{j+1}{q'}}\abs{\Delta}^{\frac{1}{q'}-1} \iint_\Delta\bigl(\M_j(\abs{f}^{q'})\bigr)^{\frac{1}{q'}}.
	\end{align}
	Combining \eqref{eq: BK v estimate} and \eqref{eq: BK f estimate}, we arrive at
	\begin{align*}
		\Bigl\vert\iint_{\R^{n+1}} \Bigl(\sum_{\Delta \in \mathcal{D}}|v_\Delta|\Bigr) f\Bigr\vert &\leq C\alpha \sum_{\Delta \in \mathcal{D}} \sum_{j\geq 1} \varepsilon_j (2^n N^2)^{\frac{j+1}{q'}} \iint_\Delta \bigl(\M_j(\abs{f}^{q'})\bigr)^{\frac{1}{q'}} \\
		&= C\alpha \sum_{j\geq 1} \varepsilon_j (2^n N^2)^{\frac{j+1}{q'}} \sum_{\Delta \in \mathcal{D}}\iint_\Delta \bigl(\M_j(\abs{f}^{q'})\bigr)^{\frac{1}{q'}} \\
		&= C\alpha \sum_{j\geq 1} \varepsilon_j (2^n N^2)^{\frac{j+1}{q'}} \iint_{\cup_{\Delta \in \mathcal{D}}\Delta} \bigl(\M_j(\abs{f}^{q'})\bigr)^{\frac{1}{q'}}.
    \end{align*}
	Next, we estimate the integral of the maximal function by Kolmogorov's inequality
    \begin{align*}
        \iint_{\cup_{\Delta \in \mathcal{D}}\Delta} \bigl(\M_j(\abs{f}^{q'})\bigr)^{\frac{1}{q'}} \leq c |\cup_{\Delta \in \mathcal{D}}\Delta|^{\frac{1}{q}} \norm{\abs{f}^{q'}}_1^{\frac{1}{q'}},
    \end{align*}
    see, e.g., \cite[Lemma 5.16]{Duoandikoetxea}. We note carefully that the constant $c$ depends only on $p$ and the weak type bound of $\M_j$. The latter depends on the dimension and the doubling constant of the metric $\d^{(j)}$, see, e.g., \cite[Lemma 3.12]{BjornBjorn} and this is just $2^{n+2}$, hence independent of $j$. Consequently, we find
    \begin{align*}
        \Bigl\vert\iint_{\R^{n+1}} \Bigl(\sum_{\Delta \in \mathcal{D}}|v_\Delta|\Bigr) f\Bigr\vert
		&\leq C\alpha \sum_{j\geq 1} \varepsilon_j (2^n N^2)^{\frac{j+1}{q'}} \abs{\cup_{\Delta \in \mathcal{D}}\Delta}^{\frac{1}{q}} \norm{\abs{f}^{q'}}_1^{\frac{1}{q'}}\\
		&\leq C\alpha^{1-\frac{p}{q}} \norm{u}_p^{\frac{p}{q}},
	\end{align*}
    with $C$ absorbing the finite quantity $S$.
	Altogether, we obtain $\sum_{\Delta \in \mathcal{D}}|v_\Delta| \in \L^q(\R^{n+1})$ with
	\begin{align}\label{eq: BK end of Step 2}
		\Bigl\|\sum_{\Delta \in \mathcal{D}}|v_\Delta|\Bigr\|_q^q \leq C\alpha^{q-p} \norm{u}_p^{p}.
	\end{align}
	
	\emph{Step 2: Bringing it all together.}
    
	We complete the proof by estimating all three terms on the right-hand side of \eqref{eq: BK decomposition of u} by the right-hand side of \eqref{eq: sufficient for weak bound BK}. By \ref{eq: CZ-i}, we have $\abs{g}\leq \min(\abs{u},\alpha)$ and so $\abs{g}^q\leq \alpha^{q-p} \abs{u}^p$, yielding
	\begin{align*}
		\norm{Tg}_q^q \leq C \norm{g}_q^q \leq C \alpha^{q-p} \norm{u}_p^p,
	\end{align*}
    where $C$ now also depends on the strong $(q,q)$ bound of $T$. By \eqref{eq: BK end of Step 2}, bounding $T\sum_{\Delta \in \mathcal{D}}v_\Delta$ is straightforward:
	\begin{align*}
		\norm{T \sum_{\Delta \in \mathcal{D}}v_\Delta}_q^q \leq C \norm{\sum_{\Delta \in \mathcal{D}}v_\Delta}_q^q \leq C \alpha^{q-p} \norm{u}_p^p.
	\end{align*}
	For $T\sum_{\Delta \in \mathcal{D}}w_\Delta$, we again proceed by duality. Let $f \in \L^{q'}(\R^{n+1} \setminus \Omega)$ with $\norm{f}_{q'} \leq 1$. As $T$ is bounded and sublinear, $T$ is also countably sublinear and we may estimate
	\begin{align*}
		\Bigl\vert\iint_{\R^{n+1} \setminus \Omega} \bigl(T\sum_{\Delta \in \mathcal{D}}w_\Delta\bigr) f\Bigr\vert
		&\leq \sum_{\Delta \in \mathcal{D}} \iint_{\R^{n+1} \setminus \Omega} \abs{T(w_\Delta) f} = \sum_{\Delta \in \mathcal{D}} \sum_{j\geq 2} \iint_{C_j^N(\Delta)} \abs{T(w_\Delta) f} \\
		&\leq \sum_{\Delta \in \mathcal{D}} \sum_{j\geq 2} \norm{T(w_\Delta)}_{\L^q(C_j^N(\Delta))} \norm{f}_{\L^{q'}(C_j^N(\Delta))}.
	\end{align*}
	By \eqref{eq: BK1}, $\norm{T(w_\Delta)}_{\L^q(C_j^N(\Delta))}$ admits exactly the same bounds as $\norm{v_\Delta}_{\L^q(C_j^N(\Delta))}$ and we may proceed as done after \eqref{eq: Step 2 duality} in Step 1 to conclude.
\end{proof}

\subsection{Extrapolation from off-diagonal estimates}

The following criterion is inspired by the arguments in \cite{KW} and learned from \cite{BechtelOuhabaz}. It can be seen as a version of the boundedness extrapolation in Lemma~\ref{lem: Toolbox} \ref{item: Extrapolation}, adapted to the weaker decay exhibited by the parabolic resolvent family in Theorem \ref{thm: parabolicODEs}. Figure~\ref{fig: extrapolation range} illustrates the extrapolation range obtained from this criterion.

\begin{prop}\label{prop: MtMx}
Let $1\le p <  s <\infty$. Let $(T_\lambda)_{\lambda>0}$ be a family of linear operators on  $\L^s(\R^{n+1})$. Assume that for sufficiently large $N>1$, there exists a constant $C$ such that for all $\lambda>0$, $j\ge 1$, and $u\in \L^p(\R^{n+1})\cap \L^s(\R^{n+1})$, the following two estimates hold:
\begin{align}
    \| \one_{\Delta_\lambda} T_\lambda (\one_{C^N_j(\Delta_\lambda)} u)\|_{s} &\le C N^{-2j} \| \one_{C^N_j(\Delta_\lambda)} u \|_s, \label{eq: MtMx1} \\
    \| T_\lambda  u\|_s &\le C \lambda^{ -\gamma_{p,s} } \|u \|_p. \label{eq: MtMx2}
\end{align}
For $\theta\in [0,1]$, set $q\coloneqq[p,s]_\theta$. If $\theta > \nicefrac{1}{q}$, then $(T_\lambda)_{\lambda>0}$ is $\L^r$ bounded for all $r\in(q,s]$.
\end{prop}

\begin{figure}[ht]
    \definecolor{gold}{HTML}{F1A226}
    \colorlet{color1}{violet}
    \colorlet{color2}{gold}
    \centering
    \begin{tikzpicture}[scale=7, font=\large]
        
        \def\s{0.15}      
        \def\p{0.7}       
        \pgfmathsetmacro{\q}{
            -\p * (\s - \p - 1)^(-1) 
        }
        \def\one{1}
        
        \filldraw[color=color2] (0,\s) circle (0.015);
        \draw[color=color2, line width=1pt] (0,\q) circle (0.015);
        \draw[color=color2, line width=1.8pt] (0,\s) -- (0,{\q-0.015});
        \draw[dashed, line width=1.8pt, color=color2] (0.02,\q) -- (\q,\q);
        \node[left, color=color2] at (0,{0.5 * (\q + \s)}) {$\nicefrac{1}{r}$};
        
        \draw[line width=1.8pt, color=color1] (0,\p) -- (1,\s);
        \node[color1, right] at (1, \s) {$\nicefrac{1}{\sigma} = \nicefrac{1}{[p,s]_\theta}$};
        
        \draw[dashed] (0,0) -- (1,1);
        
        \foreach \x in {\s,\p,1}{
            \draw[dashed, white!70!black] (0,\x) -- (1,\x);
        }
        \draw[dashed, white!70!black] (1,0) -- (1,1);
        
        \def\tick{0.012}
        
        \node[below] at (\s,0) {$\nicefrac{1}{s}$};
        \node[below] at (\q,0) {$\nicefrac{1}{q}$};
        \node[below] at (\p,0) {$\nicefrac{1}{p}$};
        \node[below] at (1,0) {$1$};
        \foreach \x in {\s,\q,\s,1}{
            \draw (\x,\tick) -- (\x,-\tick);
        }
        
        \node[left] at (-0.02,\s) {$\nicefrac{1}{s}$};
        \node[left] at (-0.02,\q) {$\nicefrac{1}{q}$};
        \node[left] at (0,\p) {$\nicefrac{1}{p}$};
        \node[left] at (0,1) {$1$};
        \foreach \y in {\s,\q,\p,1}{
            \draw (\tick,\y) -- (-\tick,\y);
        }
        
        \draw[->] 
        (0,0) -- (1.05,0) node[right] {$\theta$};
        \draw[->] 
        (0,0) -- (0,1.05) node[above] {$\nicefrac{1}{\sigma}$};
    \end{tikzpicture}
    \caption{A $(\theta, \nicefrac{1}{\sigma})$-plane covering the exponents in Proposition~\ref{prop: MtMx}. The line defined by $\nicefrac{1}{\sigma} = \nicefrac{1}{[p,s]_\theta}$ (\textcolor{color1}{violet}) and the bisector (dashed) intersect at some point $(\theta, \nicefrac{1}{q})$. Proposition \ref{prop: MtMx} extrapolates boundedness from $s$ down to but not including $q$ (\textcolor{color2}{golden}).}
    \label{fig: extrapolation range}
\end{figure}
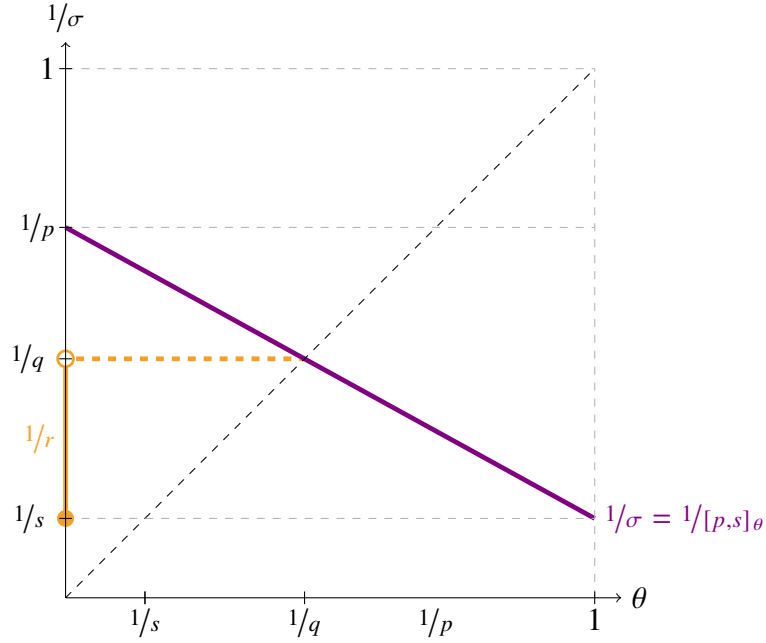
\begin{proof}
    Fix $\theta\in [0,1]$ such that for $q\coloneqq[p,s]_\theta$ we have $\theta > \nicefrac{1}{q}$. It suffices to show that there is a constant $C$ such that 
    \begin{align}\label{eq: Avg M_tM_x estimate}
        (\Avg_{s,\lambda} |(T_\lambda u)|)(x,t) &\coloneqq \left( \bariint_{\Delta_\lambda(x,t)} |(T_\lambda u)(y,s)|^s \, \d y \d s  \right )^{1/s} \leq C \left ( \M_t \M_x |u|^q\right )^{\nicefrac{1}{q}}(x,t),
    \end{align}
    holds for all $(x,t)\in\R^{n+1}$, $\lambda > 0$ and $u\in \L^p(\R^{n+1})\cap \L^s(\R^{n+1})$. Indeed, boundedness of $(T_\lambda)_{\lambda> 0}$ then readily follows from
    \begin{align*}
        \norm{T_\lambda u}_r^r 
        &= \iint_{\R^{n+1}} \abs{(T_\lambda u)(y,s)}^r \dd (y,s)\\
        &=  \iint_{\R^{n+1}} \abs{(T_\lambda u)(y,s)}^r \biggl( \bariint_{\Delta_\lambda (y,s)}\dd (x,t) \biggr) \dd (y,s) \\
        &= \iint_{\R^{n+1}} \biggl( \bariint_{\Delta_\lambda(x,t)} \abs{(T_\lambda u)(y,s)}^r \dd(y,s) \biggr) \dd(x,t) \\
        &\leq  \iint_{\R^{n+1}} \biggl( \bariint_{\Delta_\lambda(x,t)} \abs{(T_\lambda u)(y,s)}^s \dd(y,s) \biggr)^{\frac{r}{s}} \dd(x,t) \\
        &= \iint_{\R^{n+1}} \bigl(\Avg_{s,\lambda}\abs{T_\lambda u}(x,t)\bigr)^r \dd(x,t) \\
        &\leq C \iint_{\R^{n+1}} \left ( \M_t \M_x |u|^q\right )^{\frac{r}{q}} (x,t) \dd(x,t) \\
        &\leq C \norm{u}_r^r,
    \end{align*}
    where we used the $\L^{\nicefrac{r}{q}}$ bound of the iterated maximal function $\M_t\M_x$ in the penultimate step.
    
    We turn to the proof of \eqref{eq: Avg M_tM_x estimate}. By interpolation, we have
    \begin{equation}
        \| \one_{\Delta_\lambda} T_\lambda (\one_{C^N_j(\Delta_\lambda)} u)\|_{s} \le C N^{-2j\theta} \lambda^{\left ( \frac{n+2}{s}-\frac{n+2}{p} \right )(1-\theta)} \| \one_{C^N_j(\Delta_\lambda)} u \|_q,
    \end{equation}
    for all $\lambda>0$ and $u\in \L^p(\R^{n+1})\cap \L^s(\R^{n+1})$. Thus, we get for all $(x,t)\in \R^{n+1}$ and $\lambda>0$ that
    \begin{align*}
        (\Avg_{s,\lambda} |T_\lambda u|)(x,t)& = C \lambda^{-\frac{n+2}{s}} \norm{ \one_{\Delta_\lambda(x,t)} T_\lambda u}_{s}
        \\& \le C \lambda^{-\frac{n+2}{s}} \sum_{j=1}^{\infty} \| \one_{\Delta_\lambda(x,t)} T_\lambda (\one_{C^N_j(\Delta_\lambda(x,t))} u)\|_{s}
        \\& \le C \lambda^{-\frac{n+2}{s}} \sum_{j=1}^{\infty} N^{-2j\theta} \lambda^{\left ( \frac{n+2}{s}-\frac{n+2}{p} \right )(1-\theta)} \| \one_{C^N_j(\Delta_\lambda(x,t))} u \|_q
        \\&= C \lambda^{-\frac{n+2}{q}} \sum_{j=1}^{\infty} N^{-2j\theta}  \| \one_{C^N_j(\Delta_\lambda(x,t))} u \|_q
        \\&\le C \lambda^{-\frac{n+2}{q}} \sum_{j=1}^{\infty} N^{-2j\theta}  (2^{j+1}\lambda)^{\frac{n}{q}} (N^{j+1}\lambda)^{\frac{2}{q}} \left ( \M_t \M_x |u|^q\right )^{\frac{1}{q}}(x,t)
        \\& = C  \biggl( \sum_{j=1}^{\infty} N^{2j\bigl(\tfrac{1}{q}-\theta\bigr)} 2^{j\frac{n}{q}} \biggr)  \left ( \M_t \M_x |u|^q\right )^{\frac{1}{q}}(x,t) ,
    \end{align*}
    where $C$ now depends on $q$ as well. If $\theta>\nicefrac{1}{q}$, we can choose $N>1$ large enough such that the sum in $j$ converges and \eqref{eq: Avg M_tM_x estimate} follows.
\end{proof}

\section{The critical numbers}\label{S6}

Let us recall that
\begin{align*}
    & p_-(\mathcal{H})\coloneqq\inf\left\{ p\ge1: \  \left ( \mathcal{E}_\lambda \right )_{\lambda>0} \ \text{is} \ \L^p \ \text{bounded} \right\},\\
    & q_-(\mathcal{H})\coloneqq\inf\left\{ p\ge1: \ \left (\lambda \mathbb{D} \mathcal{E}_\lambda \right )_{\lambda>0} \ \text{is} \ \L^p \ \text{bounded} \right\}.
\end{align*}
The main theorem of this section is the following. 
\begin{thm}\label{thm: p_=q_}
We have $p_-(\mathcal{H})=q_-(\mathcal{H}) \in [1, 2_\star)$. More precisely, there exists $\varepsilon_0>0$ depending only on $n$ and the ellipticity constants $M$ and $\nu$ such that $p_-(\H)\le 2_\star-\varepsilon_0$.
\end{thm}
We set the stage for the proof with results that are interesting in their own right. We first recall a parabolic Sobolev embedding. For completeness, we include a proof for our choice of the parabolic gradient and the underlying energy space in Appendix~\ref{annexe}. Variants of this result are of course well-known, see, e.g., \cite[Theorem 3.1]{GopalaRao} and \cite[Lemma 3.4]{ABES2/2019}.

\begin{lem}[Parabolic Sobolev Embedding]\label{lem: Sobolev}
    Let $p \in (1, n+2)$ and $u \in \EE$ such that $\mathbb{D}u \in \L^p(\mathbb{R}^{n+1})$. Then $u \in \L^{p^\star}(\mathbb{R}^{n+1})$, and there exists a constant $C$ depending only on $n$ and $p$ such that
    \begin{align*}
        \norm{u}_{p^\star} \le C \|\mathbb{D}u\|_p.
    \end{align*}
\end{lem}
	
The following lemma serves as a blueprint for how $\L^p-\L^{p^{\star}}$ boundedness for $\left ( \mathcal{E}_\lambda \right )_{\lambda>0}$ can be obtained and why this implies an upper bound on $p_{-}(\H)$.
\begin{lem}\label{lem: 2_star}
    The family $\left ( \mathcal{E}_\lambda \right )_{\lambda>0}$ is $\L^{2_\star}-\L^2$ bounded. In particular, $p_-(\mathcal{H})\le 2_\star$.
\end{lem}
\begin{proof}
    By the parabolic Sobolev embedding above in the case $p = 2$ and the uniform boundedness of the parabolic gradient family (Proposition \ref{prop: classical ODEs}), there exists a constant $C$ depending only on $M,\nu,n$ such that for all $\lambda>0$ and $u\in \L^2(\R^{n+1})$ we have
    \begin{align*}
        \norm{\E_\lambda u}_{2^\star} \le C \norm{\mathbb{D} \E_\lambda u}_{2} 
        = C \, \lambda^{-1} \norm{\lambda \mathbb{D} \E_\lambda u}_2 
        \le C \, \lambda^{-\gamma_{2,2^\star}} \norm{u}_{2}.
    \end{align*}
    Thus, $\left( \mathcal{E}_\lambda \right)_{\lambda>0}$ is $\L^{2}-\L^{2^\star}$ bounded. The same holds for the adjoint family $\left( \mathcal{E}^\star_\lambda \right)_{\lambda>0}$ since $\H^\star$ belongs to the same class of operators as $\H$. Therefore $\left( \mathcal{E}_\lambda \right)_{\lambda>0}$ is $\L^{2_\star}-\L^2$ bounded by duality (Lemma \ref{lem: Toolbox} \ref{item: Duality principle}).
    
    By interpolation with the $\L^2$ off-diagonal estimates \eqref{eq: Classical L2 ODES} in Proposition \ref{prop: classical ODEs}, the family $\left( \mathcal{E}_\lambda \right)_{\lambda>0}$ satisfies $\L^r-\L^2$ off-diagonal estimates for all $r \in (2_\star,2]$. In particular, by Lemma \ref{lem: Toolbox} \ref{item: Extrapolation}, $\left( \mathcal{E}_\lambda \right)_{\lambda>0}$ is $\L^r$ bounded for all $r \in (2_\star,2]$, and therefore $p_-(\mathcal{H}) \le 2_\star$.
\end{proof}

The following lemma gives us the flexibility to work with higher powers of the resolvents whenever convenient. Again, the proof in \cite{AEbook2023} has nothing to do with the particular elliptic operator under consideration in this reference.

\begin{lem}[{\cite[Lemma 6.5]{AEbook2023}}]\label{lem: sans m}
Let $p \in (1,\infty)$. Assume that there exists an integer $m\ge1$ such that $(\lambda \D\mathcal{E}_\lambda^{m+1})_{\lambda>0}$ is  $\L^p$ bounded. Then so is $(\lambda \D\mathcal{E}_\lambda)_{\lambda>0}$.
\end{lem}

\begin{proof}[Proof of Theorem \ref{thm: p_=q_}] 
We proceed in three steps. The first and the last are classical, while the second makes use of all the machinery developed so far.

\emph{Step 1: Resolvent estimates from gradient bounds: $p_-(\H) \le q_-(\H)$.}

Let $p \in (q_{-}(\H),2]$. We only have to prove that $p\ge p_{-}(\H)$. Since $(\lambda \mathbb{D}\E_\lambda)_{\lambda>0}$ is $\L^{p}$ bounded, we can follow the blueprint in Lemma \ref{lem: 2_star} to deduce that $\left( \mathcal{E}_\lambda \right)_{\lambda>0}$ is $\L^{p}-\L^{p^\star}$ bounded. Interpolation with $\L^2$ off-diagonal estimates and Lemma \ref{lem: Toolbox} \ref{item: Extrapolation} again show that $\left( \mathcal{E}_\lambda \right)_{\lambda>0}$ is $\L^r$ bounded for all $r \in (p,p^\star]$, and therefore $p_-(\mathcal{H}) \le p$.

\emph{Step 2: Gradient bounds from resolvent estimates: $q_-(\H) \le p_-(\H)$.}

Let $p \in (p_{-}(\H),2]$. We extrapolate boundedness of $(\lambda\D\E_\lambda)_{\lambda>0}$ from $\L^2$ (Proposition \ref{prop: classical ODEs}) down to any $\L^s$ with $s>p$ by iterating the following scheme:
\begin{enumerate}[(1.)]
    \item Take $s \in (p,2] \cap (q_-(\H),2]$ or $s=2$.
    \item By interpolating $\L^{2_\star}-\L^2$ (Lemma \ref{lem: 2_star}) and $\L^p$ boundedness, $(\E_\lambda)_{\lambda>0}$ is $\L^\rho-\L^s$ bounded for some $\rho \in [1,s)$ and triangle interpolation (Lemma \ref{lem: gives m}) yields an integer $m\in\N$ such that $(\E_\lambda^m)_{\lambda>0}$ is $\L^p-\L^s$ bounded. 
    \item Put $T_\lambda \coloneqq \lambda \D \E_\lambda^{m+1}$. As both $(\E_\lambda)_{\lambda>0}$ and $(\lambda\D\E_\lambda)_{\lambda>0}$ are $\L^s$ bounded, $(T_\lambda)_{\lambda>0}$ is $\L^s$ bounded and the off-diagonal estimates in Theorem \ref{thm: parabolicODEs}~\ref{item: annulus to cube} for $r=\lambda$ establish \eqref{eq: MtMx1} for $(T_\lambda)_{\lambda>0}$, while the prior step takes care of \eqref{eq: MtMx2}.
    \item Proposition~\ref{prop: MtMx} yields $\L^r$ boundedness for the exponent $r$ defined in Figure~\ref{fig: figue}.
    \item Lemma~\ref{lem: sans m} yields the conclusion of 4.\ for $(\lambda\D\E_\lambda)_{\lambda>0}$ in place of $(T_\lambda)_{\lambda>0}$.
\end{enumerate}

\begin{figure}[ht]
    \definecolor{gold}{HTML}{F1A226}
    \colorlet{color1}{gold}
    \colorlet{color2}{violet}
    \centering
    \begin{tikzpicture}[scale=7, font=\large]
        
        \def\radius{0.01} 
        \def\s{0.15}      
        \def\p{0.8}       
        \pgfmathsetmacro{\q}{
            -\p * (\s - \p - 1)^(-1) 
        }
        \pgfmathsetmacro{\t}{
            1 / (2 - \s/\p)  
        }
        \pgfmathsetmacro{\r}{
            \t * \p  
        }
        
        \draw[line width=1.8pt, color=color1] (0,\p) -- (1,\s);
        \draw[line width=1.8pt, dashed, color=color1] (0,\p) -- (1,\r);
        \node[right] at (1.05, \r) {$\nicefrac{1}{\sigma} = \nicefrac{1}{[p,r]_\theta}$};
        \node[right] at (1.05, \s) {$\nicefrac{1}{\sigma} = \nicefrac{1}{[p,s]_\theta}$};
        
        \draw[dashed, line width=1.8pt, color=black] (0.02,\r) -- (1,\r);
        \draw[line width=1.8pt, color=color2] (0,0) -- (1,\p);
        \node[right] at (1.05, \p) {$\nicefrac{1}{\sigma} = \nicefrac{\theta}{p}$};
        \draw[->, dashed, black, line width=1pt]
        (1.01,\s)
        to[out=40,in=-40]
        (1.01,\r);
        
        \filldraw (\t,\r) circle (\radius); 
        
        \draw[dashed] (0,0) -- (1,1);
        
        \foreach \x in {\s,\p,1}{
            \draw[dashed] (0,\x) -- (1,\x);
        }
        \draw[dashed] (1,0) -- (1,1);
        
        \def\tick{0.012}
        
        \node[below] at (\s,0) {$\nicefrac{1}{s}$};
        \node[below] at (\q,0) {$\nicefrac{1}{q}$};
        \node[below] at (\p,0) {$\nicefrac{1}{p}$};
        \node[below] at (1,0) {$1$};
        \foreach \x in {\s,\q,\p,1}{
            \draw (\x,\tick) -- (\x,-\tick);
        }
        
        \node[left] at (-0.02,\s) {$\nicefrac{1}{s}$};
        \node[left] at (-0.02,\r) {$\nicefrac{1}{r}$};
        \node[left] at (-0.02,\q) {$\nicefrac{1}{q}$};
        \node[left] at (0,\p) {$\nicefrac{1}{p}$};
        \node[left] at (0,1) {$1$};
        \foreach \y in {\s,\r,\q,\p,1}{
            \draw (\tick,\y) -- (-\tick,\y);
        }
        
        \draw[->] 
        (0,0) -- (1.05,0) node[right] {$\theta$};
        \draw[->] (0,0) -- (0,1.05) node[above] {$\nicefrac{1}{\sigma}$};
    \end{tikzpicture}
    \caption{A $(\theta, \nicefrac{1}{\sigma})$-plane covering the exponents of the iteration scheme in Step~2 of the proof of Theorem~\ref{thm: p_=q_}. The line of exponents defined by $\nicefrac{1}{\sigma} = \nicefrac{1}{[p,s]_\theta}$ (\textcolor{color1}{golden}) and the auxiliary line $\nicefrac{1}{\sigma} = \nicefrac{\theta}{p}$ (\textcolor{color2}{violet}) intersect at some point $(\theta, \nicefrac{1}{r})$ (black) below the bisector $\theta = \nicefrac{1}{\sigma}$. The exponent $r$ belongs to $(p,2]$ by construction and to $(q_-(\H), 2]$ by step (5.)\ of the scheme. In particular, $r$ has the same properties as $s$ in step (1.) of the scheme and can serve as the starting point for the next iteration step.}
    \label{fig: figue}
\end{figure}
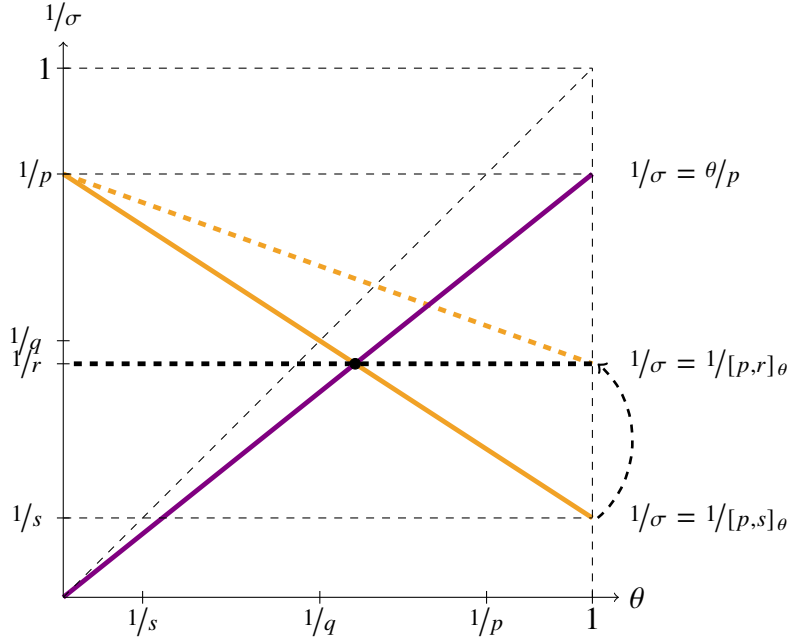

We initiate the iteration with $s_0 \coloneqq 2$ and once $s_n$ is defined for some $n \geq 0$, we set $s \coloneqq s_n$ and define $s_{n+1} \coloneqq r$ as above.
In this way, we find a decreasing sequence $(s_n)_{n\in\N}$ that converges to $p$ such that $(\lambda\D\E_\lambda)_{\lambda>0}$ is $\L^{s_n}$ bounded for every $n\in\N$. The claim follows.

\emph{Step 3: $p_-(\H) < 2_\star$}

By \cite[Lemma~2.10]{AEN2020} there exists $p \in (2,n+2)$ depending only on $n$, $M$ and $\nu$ such that $(\lambda \D \E_\lambda^\star)_{\lambda>0}$ is $\L^p$ bounded.\footnote{Since $(\lambda \D \E_\lambda^\star)_{\lambda>0}$ is a component of the operator matrix $(1+\ii \lambda P^*M)^{-1} - (1-\ii \lambda P^*M)^{-1}$ in \cite{AEN2020}.} Using the blueprint Lemma \ref{lem: 2_star} once more, we find that $(\E_\lambda^\star)_{\lambda>0}$ is $\L^{p}-\L^{p^\star}$ bounded. By duality, interpolation with $\L^2$ off-diagonal estimates and Lemma \ref{lem: Toolbox} \ref{item: Extrapolation} shows that $\left( \mathcal{E}_\lambda \right)_{\lambda>0}$ is $\L^q$ bounded for all $q \in ((p')_\star, 2]$ and therefore $p_-(\mathcal{H})\le (p')_\star < 2_\star$.
\end{proof}

\section{Boundedness of the parabolic Riesz transform}\label{S7}

In this section we prove our main results on parabolic Riesz transforms, Theorem~\ref{thm: Théorème principal}.

\subsection{Preliminaries on the functional calculus}
We recall from Theorem \ref{thm: théorie L2} that $\H$ is maximal accretive and hence sectorial with sectoriality angle at most $\frac{\pi}{2}$; in particular, it admits a holomorphic functional calculus, see, e.g., \cite{haase2006functional}. We need the following lemma for auxiliary families arising from the functional calculus. 

\begin{lem}[Functional calculus]\label{lem: funct calculus}
Let $p\in[1,2]$ and assume that $(\mathcal{E}_\lambda)_{\lambda >0}$ satisfies $\L^p$ off-diagonal estimates.
Fix $\alpha >0$ and $\beta \ge 0$ and define holomorphic functions on $\mathbb{C}\setminus (-\infty,0]$ by
\begin{align*}
\psi(z)\coloneqq z^{3\alpha}(1+z)^{-6\alpha}, \quad \varphi(z)\coloneqq (1-(1+z)^{-\beta})^{3\alpha}.
\end{align*}
Then there exists a constant $C$ such that for all $\lambda,r>0$, $u\in \L^p(\R^{n+1}) \cap \L^2(\R^{n+1})$ and all measurable sets $E,F\subseteq\mathbb{R}^{n+1}$ the following estimates hold:
\begin{enumerate}[(i)]
    \item $\displaystyle \norm{\one_F \psi(\lambda^2 \H) \varphi(r^2\H) \one_E u}_p \leq C \Bigl(1 + \frac{\d(E,F)}{\min(\lambda,r)}\Bigr)^{-6\alpha} \norm{\one_E u}_p$, \label{item: ODE for funct calc}
    \item $\displaystyle \norm{\psi(\lambda^2 \H)u} + \norm{\varphi(r^2\H) u}_p \leq C \|u\|_p$,
    \label{item: simple unif bound for funct calc}
    \item $\displaystyle \norm{\psi(\lambda^2 \H) \varphi(r^2\H) u}_p \leq C \min\Bigl(1, \, \left(\frac{r}{\lambda}\right)^{2\alpha} \Bigr) \norm{u}_p$. \label{item: boundedness for funct calc}
\end{enumerate}
In particular, these estimates hold if $p \in (p_-(\H), 2]$.
\end{lem} 

\begin{proof}
Part \ref{item: simple unif bound for funct calc} is a general principle, see, e.g., \cite[Proposition 2.6.11]{haase2006functional}. The uniform bound in \ref{item: boundedness for funct calc} follows from that. For the 'in particular part' we note that if $p \in (p_-(\H),2]$, then $\L^p$ off-diagonal estimates for $(\mathcal{E}_\lambda)_{\lambda >0}$ follow by interpolation between $\L^2$ off-diagonal estimates (Proposition~\ref{prop: classical ODEs}) and $\L^r$ boundedness for some $r < p$ (Lemma \ref{lem: Toolbox} \ref{item: Interpolation principle}). 

The remaining assertions have been obtained in the proof of \cite[Lemma 4.16]{AEbook2023} under the assumption that $((1+\varrho^2 \H)^{-1})_{\varrho \in \mathcal{U}}$ satisfies $\L^p$ off-diagonal estimates on some sector $\mathcal{U} \coloneqq \{z \in \C \setminus \{0\}: |\arg(z)| < \mu\}$ with $\mu > 0$. Thus, our task is to prove that this property follows from our weaker assumption. 

To this end, let $\lambda > 0$ and $\varrho \in \C \setminus \{0\}$. Writing 
\begin{align*}
    1+\varrho^2 \H 
    = \E_\lambda^{-1} + \Bigl(\frac{\varrho^2}{\lambda^2} -1\Bigr) \lambda^2\H
    = \E_\lambda^{-1} \biggl(1+  \Bigl(\frac{\varrho^2}{\lambda^2} -1\Bigr) (1-\E_\lambda) \biggr),
\end{align*}
we find by a Neumann series that
\begin{align}
\label{eq: Neumann series}
 (1+\varrho^2 \H)^{-1} = \sum_{k=0}^\infty (-1)^k\Bigl(\frac{\varrho^2}{\lambda^2} -1\Bigr)^k (1-\E_\lambda)^k \E_\lambda
\end{align}
provided the series converges in operator norm. By $m$-accretivity, we have $\|\E_\lambda\|_{2 \to 2} \leq 1$ and hence the series converges in $\L^2$ operator norm whenever $|\nicefrac{\varrho^2}{\lambda^2} - 1| < \nicefrac{1}{2}$. Let now $C \geq 1$ be such that $\|\E_\lambda f\|_p \leq C \|f\|_p$ for all $\lambda>0$ and all $f \in \L^p(\R^{n+1}) \cap \L^2(\R^{n+1})$. Such a constant exists since $\L^p$ off-diagonal estimates imply $\L^p$ boundedness. Under the stronger assumption $|\nicefrac{\varrho^2}{\lambda^2} - 1| < \nicefrac{1}{2(1+C)}$ we obtain from \eqref{eq: Neumann series} that
\begin{align*}
    \| (1+\varrho^2 \H)^{-1}f\|_p \leq \sum_{k=0}^\infty \biggl(\frac{1+C}{2(1+C)}\biggr)^k C \|f\|_p  = 2 C \|f\|_p.
\end{align*}
If we take $\mu \in (0, \nicefrac{\pi}{2})$ such that $|\e^{2 \ii \mu} - 1| < \nicefrac{1}{2(1+C)}$, then the above applies to $\varrho = \lambda \e^{\ii \theta}$ for every $\theta \in (-\mu,\mu)$, thereby showing that $((1+\varrho^2 \H)^{-1})_{\varrho \in \mathcal{U}}$ is $\L^p$ bounded. Stein interpolation with $\L^p$ off-diagonal estimates for $\varrho \in (0,\infty)$ yields the required $\L^p$ off-diagonal estimates on a slightly narrower sector, see, e.g., \cite[Lemma 4.13]{AEbook2023}.
\end{proof}

\subsection{Sufficient condition} 

The sufficient condition in Theorem \ref{thm: Théorème principal} is established by the following theorem.

\begin{thm}\label{thm: Riesz sufficient}
The Riesz transform $\mathcal{R}_\H$ is $\L^p$ bounded for all $p\in (p_-(\mathcal{H}),2]$.
\end{thm}

For the proof, we work with the following `kernel' representation of the Riesz transform as an improper $\L^2$-valued integral:
\begin{align}
\label{eq: Riesz transform as kernel operator}
    \mathcal{R}_\H w = c_\alpha \int_0^{\infty} \lambda \mathbb{D}\E_\lambda^{3\alpha} \psi(\lambda^2\H) w \, \frac{\d \lambda}{\lambda}, \qquad w \in \L^2(\R^{n+1}).
\end{align}
Here, $\psi$ is as in Lemma~\ref{lem: funct calculus} and the parameter $\alpha>0$ implicit in its definition is still at our disposal and $c_\alpha>0$ is a constant. One of the key arguments in our proof will be to take $\alpha$ large, as this will help with the decay of the kernel. The representation \eqref{eq: Riesz transform as kernel operator} is obtained from the Calder\'on reproducing formula 
\begin{align*}
        w = c_\alpha \int_0^{\infty} (\lambda^2 \H)^{1/2}(1+\lambda^2 \H)^{-3\alpha}(\lambda^2 \H)^{3\alpha}(1+\lambda^2 \H)^{-6\alpha} w \, \frac{\d\lambda}{\lambda},
\end{align*}
see, e.g., \cite[Theorem 6.16]{egert2024harmonic} by applying the $\L^2$ bounded operator $\mathcal{R}_\H$ to both sides.

For $p$ as in our main theorem, we have the following off-diagonal bounds for (parts of) the kernel.

\begin{lem}\label{lem: beta extrapolation}
For all $p\in (p_-(\mathcal{H}),2]$, there exists an integer $\beta\ge1$ such that $(\E^\beta_\lambda)_{\lambda>0}$ satisfies $\L^p-\L^2$ off-diagonal estimates.
\end{lem}
\begin{proof}
We fix $\varrho_1, \varrho_2 \in (p_-(\mathcal{H}), p)$ such that $\varrho_1 < \varrho_2$. Since $(\E_\lambda)_{\lambda>0}$ is $\L^2$, $\L^{\varrho_1}$, and $\L^{2_\star}-\L^2$ bounded, it follows from Lemma~\ref{lem: gives m} that there exists an integer $\beta \ge 1$ such that $(\E^{\beta}_\lambda)_{\lambda>0}$ is $\L^{\varrho_2}-\L^2$ bounded. The claim follows by interpolation with the $\L^2$ off-diagonal estimates for $(\E^{\beta}_\lambda)_{\lambda>0}$ (Proposition \ref{prop: classical ODEs} and Lemma \ref{lem: Toolbox} \ref{item: Interpolation principle}).
\end{proof}

\begin{proof}[Proof of Theorem \ref{thm: Riesz sufficient}]
    We shall employ an iteration scheme. Namely, we will show that whenever $q\in(p_{-}(\H),2]$ is such that $\RH$ is $\L^q$ bounded, then $\RH$ is $\L^p$ bounded for all $p \in (p_{-}(\H), q) \cap (q_\star, q)$. In this way, $\L^2$ boundedness of $\RH$ yields $\L^p$ boundedness for all $p \in (2_\star,2]$ and we obtain $\L^p$ boundedness for every exponent $p \in (p_{-}(\H),2_\star]$ after finitely many iteration steps. 

    For the rest of the proof, we pick $q$ and $p$ as above. By definition, we have that the resolvents $(\E_\lambda)_{\lambda>0}$ are $\L^p$ bounded and Theorem~\ref{thm: p_=q_} yields that so are their gradients $(\lambda \D \E_\lambda)_{\lambda>0}$.

    We shall verify the assumptions of the Blunck--Kunstmann extrapolation theorem, Theorem~\ref{thm: BK}. To this end, we fix $r>0$ and $u \in \Z$ with support in a parabolic cube $\Delta_r$ and we let $N\in\N$ subject to further specification in the bulk of the proof. The argument comes in three steps.
	
	\emph{Step 1: Decomposing $u$.}
	
    By Lemma \ref{lem: beta extrapolation}, there exists an integer $\beta\ge1$ such that $(\E^\beta_\lambda)_{\lambda>0}$ satisfies $\L^{p}-\L^2$ off-diagonal estimates. By interpolation with $\L^{p}$ boundedness, the family $(\E^\beta_\lambda)_{\lambda>0}$ satisfies $\L^{p}-\L^q$ off-diagonal estimates.
    
    For an integer $\alpha$ to be chosen later and the above choice of $\beta$ we define
    \begin{align*}
        \varphi(z)\coloneqq\big ( 1-(1+z)^{-\beta} \big )^{3\alpha}, \quad z \in \C\setminus \left\{-1\right\}
    \end{align*}
    and decompose $u$ via
    \begin{align*}
        u = (1-\varphi(r^2\H)) u + \varphi(r^2\H) u \eqqcolon v + w.
    \end{align*}
    Since $p_{-}(\H) < q \leq 2$ and $u \in \Z \subseteq \L^q(\R^{n+1}) \cap \L^2(\R^{n+1})$, we have $v,w \in \L^q(\R^{n+1}) \cap \L^2(\R^{n+1})$.
    
    \emph{Step 2: Verification of \eqref{eq: BK2}.}
    
    Expanding $v$ yields
    \begin{align*}
        v = u -\sum_{k=0}^{3\alpha} \binom{3\alpha}{k} (-1)^k \E^{k\beta}_{r} u = -\sum_{k=1}^{3\alpha} \binom{3\alpha}{k} (-1)^k \E^{k\beta}_{r} u.
    \end{align*}
   Each $\E^{k\beta}_{r}$, $k \geq 1$, satisfies $\L^p-\L^q$ off-diagonal estimates. Indeed, for $k = 1$ this is due to our choice of $\beta$ in Step~1, and for $k \ge 2$, we have $\E^{k\beta}_{r} = \E^{(k-1)\beta}_{r} \E_r^\beta$, where $\E^{(k-1)\beta}_{r}$ satisfies $\L^{q}$ off-diagonal estimates by interpolating $\L^2$ off-diagonal estimates with $\L^p$ boundedness. Thus, for all $j\ge 1$ we have
    \begin{align*}
        \norm{\one_{C^N_j(\Delta_r)}v}_q &\leq \sum_{k=1}^{3\alpha} \binom{3\alpha}{k} \norm{\one_{C^N_j(\Delta_r)} \E^{k\beta}_{r} u}_q
        \\&\leq C \sum_{k=1}^{3\alpha} \binom{3\alpha}{k} \e^{-c\frac{2^{j-1}r}{r}} r^{-\gamma_{p,q}} \norm{u}_{p}
        \\& \leq C \e^{-c2^{j-1}} r^{-\gamma_{p,q}} \norm{u}_{p},
    \end{align*}
    where $C$ and $c$ depend on $p, q, \alpha, \beta, \nu, M$. This is \eqref{eq: BK2} with $\eps_j \coloneqq C \e^{-c2^{j-1}}$ and $\sum_{j \ge 1} \eps_j \big( 2^n N^2 \big)^{\frac{j}{q'}}$ is finite, no matter the choice of $N$.

    \emph{Step 3: Verification of \eqref{eq: BK1}.}
    
    From this point on we impose the condition $3\alpha \ge \beta+1$ and write \eqref{eq: Riesz transform as kernel operator} as
    \begin{align*}
        \RH w & = c_\alpha \int_{0}^{\infty} \lambda \mathbb{D}\E_\lambda^{3\alpha-\beta} \E^{\beta}_\lambda \psi(\lambda^2\H) \varphi(r^2\H) u \, \frac{\d \lambda}{\lambda}
        \\&= c_\alpha \int_{0}^{\infty} \lambda \mathbb{D}\E_\lambda^{3\alpha-\beta} \bigl( \one_{E}+\one_{E^\c} \bigr) \E^{\beta}_\lambda \psi(\lambda^2\H) \varphi(r^2\H) u \, \frac{\d \lambda}{\lambda},
    \end{align*}
    where $E \coloneqq \Delta^{j-\nicefrac{3}{2},\nicefrac{j}{4}}_r$ is as in Remark \ref{rem: epsilon2}. Thus, for all $j \ge 2$ we have
    \begin{align*}
        \norm{\one_{C^N_j(\Delta_r)} \RH w}_q &\le c_\alpha \int_{0}^{\infty} \norm{\one_{C^N_j(\Delta_r)}\lambda \mathbb{D}\E_\lambda^{3\alpha-\beta} \one_{E} \E^{\beta}_\lambda \psi(\lambda^2\H)\varphi(r^2\H) u}_q \, \frac{\d \lambda}{\lambda} \\
        & \qquad + c_\alpha \int_{0}^{\infty} \norm{\one_{C^N_j(\Delta_r)}\lambda \mathbb{D}\E_\lambda^{3\alpha-\beta} \one_{E^\c} \E^{\beta}_\lambda \psi(\lambda^2\H)\varphi(r^2\H) u}_q \, \frac{\d \lambda}{\lambda} \\
        &\eqqcolon \mathrm{I}+\mathrm{II}.
    \end{align*}
    
    In order to estimate $\mathrm{I}$, we use in sequence: First, $\L^q$ off-diagonal estimates for the gradients as in Theorem \ref{thm: parabolicODEs} \ref{item: cube to annulus} (see also Remark \ref{rem: epsilon2}), second, $\L^p-\L^q$ boundedness for the $\beta$th powers of the resolvents (see Step~1), and third Lemma \ref{lem: funct calculus} \ref{item: boundedness for funct calc}. This yields
    \begin{align}
    \label{eq: THE ONE}
    \begin{split}
        \mathrm{I} &\leq C N^{-j\bigl(1+\frac{1}{1+q'}\bigr)} \int_{0}^{\infty} \bigg(\frac{\lambda}{r} +\Big(\frac{\lambda}{r}\Big)^{4N}\bigg)\| \E^{\beta}_\lambda \psi(\lambda^2\H)\varphi(r^2\H) u\|_{q} \, \frac{\d \lambda}{\lambda}
        \\& \leq C N^{-j\bigl(1+\frac{1}{1+q'}\bigr)} \int_{0}^{\infty} \bigg(\frac{\lambda}{r} +\Big(\frac{\lambda}{r}\Big)^{4N}\bigg) \lambda^{-\gamma_{p,q}} \| \psi(\lambda^2\H)\varphi(r^2\H) u\|_{p} \, \frac{\d \lambda}{\lambda}
        \\& \leq C N^{-j\bigl(1+\frac{1}{1+q'}\bigr)} \int_{0}^{\infty} \bigg(\frac{\lambda}{r} +\Big(\frac{\lambda}{r}\Big)^{4N}\bigg) \lambda^{-\gamma_{p,q}} \min\Big(1,\left(\frac{r}{\lambda} \right)^{2\alpha}\Big)\norm{u}_{p} \, \frac{\d \lambda}{\lambda}
        \\& \leq C N^{-j\bigl(1+\frac{1}{1+q'}\bigr)} \bigg (  \int_{0}^{\infty} (\lambda +\lambda^{4N}) \lambda^{-\gamma_{p,q}} \min(1,\lambda^{-2\alpha}) \, \frac{\d \lambda}{\lambda} \bigg ) r^{-\gamma_{p,q}} \norm{u}_{p},
    \end{split}
    \end{align}
    where $C$ now also depends on $N$ but not on $j$.
    
    As for $\mathrm{II}$, we set $F \coloneqq \Delta^{j-\nicefrac{3}{4},\nicefrac{1}{8}}_r$, and use that the resolvents and their gradients are $\L^q$ bounded in order to obtain
    \begin{align}
    \begin{split}
        \label{eq: THE TWO}
        \mathrm{II} &\leq \int_{0}^{\infty} \norm{\one_{E^\c} \E^{\beta}_\lambda \left( \one_{F}+\one_{F^\c} \right) \psi(\lambda^2\H)\varphi(r^2\H) u}_q \, \frac{\d \lambda}{\lambda}
        \\& \leq \int_{0}^{\infty} \norm{\one_{E^\c} \E^{\beta}_\lambda \one_{F}\psi(\lambda^2\H)\varphi(r^2\H) u}_q \, \frac{\d \lambda}{\lambda}
        \\& \qquad+ \int_{0}^{\infty} \norm{\one_{E^\c} \E^{\beta}_\lambda \one_{F^\c}  \psi(\lambda^2\H)\varphi(r^2\H) u}_q \, \frac{\d \lambda}{\lambda}
        \\&\eqqcolon\mathrm{II}_1+\mathrm{II}_2.
    \end{split}
    \end{align}
    We continue to estimate $\mathrm{II}_1$ and $\mathrm{II}_2$ separately.
    
    Starting with $\mathrm{II}_1$, we require from now on that $N\geq 2^8$. Now, $\L^p-\L^q$ off-diagonal estimates and the $\L^p$ bounds for the functional calculus families in Lemma~\ref{lem: funct calculus} \ref{item: boundedness for funct calc} yield
    \begin{align}\label{eq: THE TWO one}
    \begin{split}
        \mathrm{II}_1 &\leq C \int_{0}^{\infty} \e^{-c\frac{2^jr}{\lambda}} \lambda^{-\gamma_{p,q}} \norm{\psi(\lambda^2\H)\varphi(r^2\H) u}_p \, \frac{\d \lambda}{\lambda}
        \\&\leq C \int_{0}^{\infty} \e^{-c\frac{2^jr}{\lambda}} \lambda^{-\gamma_{p,q}} \min\Big(1,\left(\frac{r}{\lambda} \right)^{2\alpha}\Big)\norm{ u}_p \, \frac{\d \lambda}{\lambda}
        \\& \leq C 2^{-j\left(\gamma_{p,q}+2\alpha\right)} \left (  \int_{0}^{\infty} \e^{-\frac{c}{\lambda}} \lambda^{-(\gamma_{p,q}+2\alpha)}  \, \frac{\d \lambda}{\lambda} \right ) r^{-\gamma_{p,q}} \norm{ u}_p.
    \end{split}
    \end{align}
    For $\mathrm{II}_2$, we first use $\L^p-\L^q$ boundedness to give
    \begin{align*}
        \mathrm{II}_2 \leq C \int_{0}^{\infty} \lambda^{-\gamma_{p,q}} \norm{\one_{F^\c}\psi(\lambda^2\H)\varphi(r^2\H) u}_p \, \frac{\d \lambda}{\lambda}.
    \end{align*}
    By Lemma~\ref{lem: funct calculus}~\ref{item: ODE for funct calc} we have that
    \begin{align*}
        \norm{\one_{F^\c}\psi(\lambda^2\H)\varphi(r^2\H) u}_p \leq C 
            \left(1+c\frac{2^{j}r}{\lambda} \right)^{-6\alpha} \norm{u}_p 
    \end{align*}
    where $c$ now also depends on $N$. Taking a geometric mean with the bound provided by Lemma~\ref{lem: funct calculus}~\ref{item: boundedness for funct calc}, we conclude 
    \begin{align*}
        \norm{\one_{F^\c}\psi(\lambda^2\H)\varphi(r^2\H) u}_p \leq \left(\frac{r}{\lambda}\right)^{\alpha} \bigg(1+c\frac{2^{j}r}{\lambda} \bigg)^{-3\alpha} \norm{u}_p.
    \end{align*}
    Altogether, we obtain
    \begin{align}
    \label{eq: THE TWO two}
    \begin{split}
        \mathrm{II}_2 &\leq C \bigg( \int_{0}^{\infty}  \lambda^{-\gamma_{p,q}} \left(\frac{r}{\lambda}\right)^{\alpha} \bigg(1+c\frac{2^{j}r}{\lambda} \bigg)^{-3\alpha}  \, \frac{\d \lambda}{\lambda} \bigg) \norm{u}_p
        \\&\leq C 2^{-j\left( \gamma_{p,q}+\alpha \right)} \bigg( \int_{0}^{\infty}  \frac{\lambda^{\alpha+\gamma_{p,q}}}{\left(1+c\lambda\right)^{3\alpha}}  \, \frac{\d \lambda}{\lambda} \bigg) r^{-\gamma_{p,q}} \norm{u}_p.
    \end{split}
    \end{align}
    Collecting the estimates \eqref{eq: THE ONE} - \eqref{eq: THE TWO two}, we obtain
    \eqref{eq: BK1} with $\eps_j \coloneqq C(N^{-j(1+\frac{1}{1+q'})} + 2^{-j(\gamma_{p,q}+\alpha)})$, where $C$ contains the various remaining numerical integrals in $\lambda$ and guaranteeing that $C$ is actually finite is one of our remaining tasks. This will be possible by taking $\alpha$ and $N$ even larger. 
    
    Indeed, finiteness of $C$ requires that the three integrals 
    \begin{align}
    \label{eq: THE INTEGRALS}
    \begin{split}
        &\int_{0}^{\infty} \big(\lambda + \lambda^{4N}\big) \lambda^{-\gamma_{p,q}} \min\big(1, \lambda^{-2\alpha}\big) \, \frac{\d \lambda}{\lambda}, \qquad 
        \int_{0}^{\infty} \e^{-\frac{c}{\lambda}} \lambda^{-(\gamma_{p,q}+2\alpha)}  \, \frac{\d \lambda}{\lambda} \\
        &\text{and} \qquad \int_{0}^{\infty}  \frac{\lambda^{\alpha+\gamma_{p,q}}}{\left(1+c\lambda\right)^{3\alpha}}  \, \frac{\d \lambda}{\lambda}
    \end{split}
    \end{align}
    are finite, which can all be accomplished by picking $\alpha \geq 2N +1$ large enough and using that the first integral converges at $0$ since $p\in (q_\star,q)$ and therefore
    \begin{align*}
        - \gamma_{p,q} > - \gamma_{q_\star,q} = -1.
    \end{align*}
    With $C$ being finite, summability as required in Theorem~\ref{thm: BK} means that the series
    \begin{align}
    \label{eq: THE SERIES}
    \begin{split}
        \sum_{j\geq1} N^{-j\bigl(1+\frac{1}{1+q'}\bigr)} (2^nN^2)^{\nicefrac{j}{q'}}\qquad \text{and} \qquad \sum_{j\geq1} 2^{-j\left( \gamma_{p,q} + \alpha \right)} (2^nN^2)^{\nicefrac{j}{q'}}
    \end{split}
    \end{align}
    need to converge. The first one converges by picking $N$ large enough since $1+\nicefrac{1}{1+q'} > \nicefrac{2}{q'}$. The second one converges by picking $\alpha$ large once $N$ is fixed.
\end{proof}

\subsection{Necessary condition}

The necessary condition in Theorem \ref{thm: Théorème principal} is established by the following proposition.

\begin{prop}\label{prop: Necessary}
    If $p\in (1,2)$ is such that the Riesz transform $\mathcal{R}_\H$ is $\L^p$ bounded, then $p\ge p_-(\mathcal{H})$.
\end{prop}

\begin{proof}
    We set $p_0 \coloneqq p$ and iteratively define $p_k \coloneqq p^\star_{k-1}$, stopping at the first index $k^+ \ge 0$ for which $p_{k^+} \in [2_\star, 2)$. We prove by backward induction that $p_k \ge p_{-}(\H)$ holds for all $0 \le k \le k^+$.

    For $k = k^+$, we have $p_{k^+} \ge 2_\star \ge p_{-}(\H)$ by Lemma \ref{lem: 2_star}.

    Let $1 \le k \le k^+$ and assume that $p_k \ge p_-(\H)$. We now prove that $p_{k-1} \ge p_-(\H)$. Fix $q \in (p_{k-1}, 2_\star)$. Then we have $q^\star \in (p_k, 2]$, and hence $q^\star > p_-(\H)$. Lemma \ref{lem: funct calculus} \ref{item: simple unif bound for funct calc} with $\alpha =\nicefrac{1}{6}$ ensures that $(\lambda \H^{1/2} \E_\lambda)_{\lambda>0}$ is $\L^{q^\star}$ bounded. For $u\in \L^q(\R^{n+1}) \cap \L^2(\R^{n+1})$ we also have $u\in \operatorname{ran}(\sqrt{\mathcal{H}})$ by Lemma \ref{lem: many functions}. Using composition rules for the functional calculus (e.g.\ \cite[Proposition 5.15]{egert2024harmonic}), boundedness of $\RH$, and $(\lambda \H^{1/2} \E_\lambda)_{\lambda>0}$ and the parabolic Sobolev embedding (see Lemma \ref{lem: Sobolev}), we obtain
    \begin{align*}
        \norm{\E_\lambda u}_{q^\star} = \lambda^{-1} \norm{\lambda \H^{1/2} \E_\lambda \H^{-1/2}u}_{q^\star} &\leq C \lambda^{-1} \norm{ \H^{-1/2}u}_{q^\star}
        \\& \leq C \lambda^{-1} \norm{ \mathbb{D} \H^{-1/2}u}_{q}
        \\& \leq C \lambda^{-\gamma_{q,q^\star}} \norm{ u}_{q}
    \end{align*}
    for all $\lambda > 0$. Hence,  $(\E_\lambda)_{\lambda>0}$ is $\L^{q} - \L^{q^\star}$ bounded. As this is true for all $q\in (p_{k-1},2_\star)$, by interpolation with the $\L^2$ off-diagonal estimates in Proposition \ref{prop: classical ODEs} we find for all $q\in (p_{k-1},2_\star)$ that there exists $q^\circ \in (q,2)$ such that $\left( \mathcal{E}_\lambda \right)_{\lambda>0}$ satisfies $\L^{q} - \L^{q^\circ}$ off-diagonal estimates. In particular, by Lemma \ref{lem: Toolbox} \ref{item: Extrapolation}, $\left( \mathcal{E}_\lambda \right)_{\lambda>0}$ is $\L^q$ bounded for all $q \in (p_{k-1},2]$, and therefore $p_{k-1} \ge p_-(\H)$.
\end{proof}

\section{Improvements for real-valued coefficients}\label{S8}

In this section, we consider the case that the matrix-valued function $A$ has real-valued coefficients. The main result is the following theorem, yielding $\L^p$ boundedness of the parabolic Riesz transform in the full range $p \in (1,2]$.

\begin{thm}\label{thm: A real}
If the coefficients of $A$ are real-valued, then $p_-(\mathcal{H}) = 1$. In particular, for every $p \in (1,2]$, the Riesz transform $\mathcal{R}_\H=(\nabla_x \mathcal{H}^{-1/2}, D_t^{1/2} \mathcal{H}^{-1/2})$ is bounded on $\L^p(\mathbb{R}^{n+1})$. 
\end{thm}

The key point here is the availability of Gaussian bounds, which we state next.

\begin{lem}[{\cite[Lemma 4.3]{AEN2018}}]\label{lem: bornes gaussiennes}
Suppose the coefficients of $A$ are real. For every $\lambda >0$ and $m \geq 1$, the resolvent $\mathcal{E}_\lambda^m$ is an integral operator on $\L^2(\R^{n+1})$ with kernel $K^m_\lambda$ satisfying the following pointwise estimate:
\begin{equation}\label{eq: real kernel}
|K^m_\lambda(x,t;y,s)|
\le C \,\mathbf{1}_{(0,\infty)}(t-s)
\frac{\e^{-\frac{t-s}{\lambda^2}}}{\lambda^{2m}}
\frac{\e^{-c\frac{|x-y|^2}{t-s}}}{(t-s)^{n/2-(m-1)}},
\end{equation}
where $C$ and $c$ are constants depending only on $M$, $\nu$, $n$, and $m$.
\end{lem}

These kernel bounds imply off-diagonal estimates for the resolvent.

\begin{lem}\label{lem: hors diagonal grace aux gaussiennes}
  Suppose the coefficients of $A$ are real. For every $p \in [1,2]$ there exists $m \geq 1$ such that $(\mathcal{E}_\lambda^m)_{\lambda>0}$ satisfies $\L^1-\L^p$ off-diagonal estimates. If $p=1$, we can take $m=1$. 
\end{lem}

\begin{proof}
Let $u \in \L^1(\R^{n+1}) \cap \L^2(\R^{n+1})$, let $E,F \subseteq \R^{n+1}$ be measurable and set $\d \coloneqq \d(E,F) \geq 0$. For $\lambda>0$ the kernel bounds in Lemma~\ref{lem: bornes gaussiennes}, together with Young's convolution inequality, yield
\begin{align*}
    \norm{\one_F \E_\lambda^m (\one_E u)}_p &\le C  \bigg( \iint_{\R^{n+1}} \one_{(0,\infty)}(t) \frac{\e^{-p\frac{t}{\lambda^2}}}{\lambda^{2pm}} \frac{\e^{-pc \frac{|x|^2}{t}}}{t^{p(\nicefrac{n}{2}-(m-1))}} \one_{[0,\infty)}(|x|^2+t - \d^2) \, \mathrm{d}x \, \mathrm{d}t \bigg)^{1/p} \norm{\one_E u}_1
    \\& =  \frac{C}{\lambda^{2m}} \biggl( \int^{\infty}_{0}  \frac{\e^{-p\frac{t}{\lambda^2}}}{t^{p(\nicefrac{n}{2}-(m-1))-\frac{n}{2}}}  \bigg( \int_{|x|^2\ge \frac{\d^2}{t}-1}\e^{-pc |x|^2} \, \mathrm{d}x \bigg) \, \mathrm{d}t \biggr)^{1/p} \norm{\one_E u}_1
    \\& \leq\e^{c/2} \left( \frac{2\pi}{p c} \right)^{\frac{n}{2p}} \frac{C}{\lambda^{2m}} \biggl( \int^{\infty}_{0}  \frac{\e^{-p\frac{t}{\lambda^2}}\e^{-\frac{pc}{2} \frac{\d^2}{t}}}{t^{p(\nicefrac{n}{2}-(m-1))-\frac{n}{2}}}  \, \mathrm{d}t \biggr)^{1/p} \norm{\one_E u}_1
    \\&= C\e^{c/2} \left( \frac{2\pi}{p c} \right)^{\frac{n}{2p}} \lambda^{-\gamma_{1,p}} \biggl( \int^{\infty}_{0}  \frac{\e^{-ps}\e^{-\frac{pc}{2} \frac{\d^2}{\lambda^2} \frac{1}{s}}}{s^{p(\nicefrac{n}{2}-(m-1))-\frac{n}{2}}}  \, \mathrm{d}s \biggr)^{1/p} \norm{\one_E u}_1,
\end{align*}
where we used a change of variables in $x$ in the second line and in $t$ in the last line, and wrote $\e^{-pc|x|^2}=\e^{-\nicefrac{pc}{2}|x|^2}\e^{-\nicefrac{pc}{2}|x|^2}$ in the second line to bound the expression by a Gaussian integral. Note that the term $\one_{[0,\infty)}(|x|^2+t - \d^2)$ accounts for the (possible) separation of $E$ and $F$. We then choose $m \ge 1$ such that $p(\nicefrac{n}{2}-(m-1))-\frac{n}{2} < 1$, so that the above integral is finite (integrability near $0$ when $\d = 0$). The result then follows immediately by writing
\begin{align*}
    \int^{\infty}_{0}  \frac{\e^{-ps}\e^{-\frac{pc}{2} \frac{\d^2}{\lambda^2} \frac{1}{s}}}{s^{p(\nicefrac{n}{2}-(m-1))-\frac{n}{2}}}  \, \mathrm{d}s &= \int^{\nicefrac{\d}{\lambda}}_{0}  \frac{\e^{-ps}\e^{-\frac{pc}{2} \frac{\d^2}{\lambda^2} \frac{1}{s}}}{s^{p(\nicefrac{n}{2}-(m-1))-\frac{n}{2}}}  \, \mathrm{d}s + \int^{\infty}_{\nicefrac{\d}{\lambda}}  \frac{\e^{-ps}\e^{-\frac{pc}{2} \frac{\d^2}{\lambda^2} \frac{1}{s}}}{s^{p(\nicefrac{n}{2}-(m-1))-\frac{n}{2}}} \, \mathrm{d}s \\&\le  \biggl( \int^{\infty}_{0}  \frac{\e^{-\frac{p}{2}s} }{s^{p(\nicefrac{n}{2}-(m-1))-\frac{n}{2}}}  \, \mathrm{d}s \biggr)\e^{-\min(\frac{pc}{2},\frac{p}{2}) \frac{\d}{\lambda}},
\end{align*}
where the last estimate follows from elementary inequalities. If $p=1$, then the condition on $m$ allows us to take $m=1$.
\end{proof}

\begin{proof}[Proof of Theorem \ref{thm: A real}] 
Lemma~\ref{lem: hors diagonal grace aux gaussiennes} yields that $(\mathcal{E}_\lambda)_{\lambda>0}$ is $\L^1$ bounded, and therefore $p_-(\mathcal{H}) = 1$. In particular, by Theorem \ref{thm: Riesz sufficient}, for every $p \in (1,2]$, the Riesz transform $\mathcal{R}_\H$ is bounded on $\L^p(\mathbb{R}^{n+1})$.
\end{proof}

For the first component $\nabla_x \H^{-\nicefrac{1}{2}}$, the previous result can be improved and we get the weak type bound at the endpoint $p=1$. It is an open problem whether the same holds true for the full Riesz transform. 

\begin{thm}\label{thm: weak-type (1,1)}
    If the coefficients of $A$ are real, then $\nabla_x \mathcal{H}^{-1/2}$ is of weak type $(1,1)$. 
\end{thm}

\begin{proof}
The claim follows by revisiting the proof of Theorem~\ref{thm: Riesz sufficient} and feeding in the better estimates that are available in the present setting. We take $p=1$ and $q=2$ in that proof. No iteration is needed. Moreover, we take $N=2$ since in the absence of the half-order time derivative $D_t^{1/2}$ no off-diagonal estimates with time-stretching enter our estimates.

Let us recapitulate the ingredients needed to repeat the proof of Theorem~\ref{thm: Riesz sufficient} \emph{verbatim} with these choices of parameters:
\begin{enumerate}[(1.)]
    \item The estimates of Lemma~\ref{lem: funct calculus} with $p=1$,
    \item $\L^1-\L^2$ off-diagonal estimates for $(\E_\lambda^\beta)_{\lambda >0}$ for $\beta \in \N$ large enough,
    \item $\L^2$ off-diagonal estimates for $(\nabla_x \E_\lambda^m)_{\lambda>0}$ for all $m \in \N$.
\end{enumerate}
Now, (1.) and (2.) are a consequence of Lemma~\ref{lem: hors diagonal grace aux gaussiennes} and (3.) is due to Proposition~\ref{prop: classical ODEs}. Moreover, since the off-diagonal estimates in (3.) are the standard (exponential) ones rather than the more involved ones from Theorem~\ref{thm: parabolicODEs}, the upper bound of $\mathrm{I}$ in \eqref{eq: THE ONE} becomes the same as the upper bound for $\mathrm{II}$ in \eqref{eq: THE TWO} and therefore only the second series in \eqref{eq: THE SERIES} and the second and third integrals in \eqref{eq: THE INTEGRALS} need to be made finite. This can be achieved by taking $\alpha$ large alone.
\end{proof}

\begin{rem}
The proofs of Theorems~\ref{thm: A real} and \ref{thm: weak-type (1,1)} have used the assumption that the coefficients of $A$ are real only through the kernel bounds of Lemma~\ref{lem: bornes gaussiennes}. These bounds in turn are a direct consequence of Gaussian upper bounds for the fundamental solution of $\H$, see \cite[Lemma 4.3]{AEN2018}. Thus, both Theorem~\ref{thm: A real} and Theorem~\ref{thm: weak-type (1,1)} hold \emph{in extenso} if Gaussian upper bounds for the fundamental solution of $\H$ are available. Besides the case of real-valued coefficients, this applies to small complex perturbations of real coefficients~\cite{hofmann2004gaussian}, and autonomous operators with coefficients $A = A(x)$ when $n=1,2$. In the autonomous setting, the fundamental solution is the semigroup generated by the elliptic part of $\H$ and Gaussian upper bounds have been obtained in \cite[Theorem 2.21]{AMT1998} and \cite[Theorem 3.5]{AMT1998}.
\end{rem}

\section{Sharpness in spatial dimension $n \geq 2$}\label{S9}

In this short section, we prove that our main results for $n \geq 2$ are sharp within the class of all parabolic operators $\H$ in the following sense. 

\begin{prop}\label{prop: Sharpness}
Assume that $n\ge 2$. For every $\varepsilon>0$, there exists a bounded and uniformly elliptic matrix-valued function $A_\varepsilon\in \L^\infty(\R^{n+1};M_n(\C))$ such that, with $\H_\varepsilon=\partial_t-\mathrm{div}_x(A_\varepsilon \nabla_x)$, one has
\begin{equation*}
    2_\star-\varepsilon \leq p_-(\H_\varepsilon).
\end{equation*}
\end{prop}

\begin{cor}
Assume that $n\ge 2$. For every $p\in [1,2_\star)$, there exists a parabolic operator $\H$ such that the Riesz transform $\mathcal{R}_\H=(\nabla_x \mathcal{H}^{-1/2}, D_t^{1/2}\mathcal{H}^{-1/2})$ is not bounded on $\L^p(\mathbb{R}^{n+1})$.
\end{cor}

\begin{proof}
    This is an immediate consequence of Propositions \ref{prop: Necessary} and \ref{prop: Sharpness}. 
\end{proof}

The proof of Proposition \ref{prop: Sharpness} relies on Mooney's irregular weak solutions to uniformly parabolic equations from \cite[Section 2]{Mooney2021}, following arguments similar to those in \cite{BMV24}. For the convenience of the reader, we give a brief summary of the construction and the precise properties that will be needed for our purpose.

According to \cite[Theorem 2.2]{Mooney2021}, for every $\mu\in[0,\nicefrac{n}{2})$ there exists a uniformly elliptic matrix $A = A(x)$ such that the equation $\Div_x(A \nabla_x w)=\frac{1}{2}(\ii w+\mu w+x\cdot \nabla_x w)$ has a non-trivial weak solution $w$, where $w$ and $A$ are smooth outside the origin \cite[Remark 3.4]{Mooney2021} and are such that for $\abs{x} \geq 1$ we have the pointwise bounds
\begin{align}
\label{eq: Mooney-bounds}
	\abs{w(x)} \leq C \abs{x}^{-\mu} \quad \text{and} \quad	\abs{\partial_i A(x)} \leq C \abs{x}^{-1},
\end{align}
see \cite[p.~203]{BMV24}. Then, the function
\begin{align*}
	u(x,t) \coloneqq (-t)^{-\nicefrac{\mu}{2}} \e^{-\frac{\ii}{2}\log(-t)} w((-t)^{-\nicefrac{1}{2}}x)
\end{align*}
is smooth up to $t=0$ away from $x=0$ and a weak solution to the parabolic equation $\partial_t u- \Div_x(A \nabla_x u) = 0$ on $\R^n \times (-\infty,0)$ with coefficients $A = A((-t)^{-\nicefrac{1}{2}}x)$. It can be extended to a weak solution to a parabolic equation on all of $\R^{n+1}$ by solving the heat equation for suitable matching initial data at $t=0$ \cite{Mooney2021}. Moreover, given any $q>2^\star$ we can enforce
\begin{align}
\label{eq: Mooney-notLp}
    \norm{u}_{\L^{q}(B(0,\nicefrac{1}{2})\times(-\nicefrac{1}{2},0))}=\infty
\end{align}
by taking $\mu$ sufficiently close to $\nicefrac{n}{2}$. This follows either from a direct computation or from~\cite[p.~205]{BMV24}; note that \cite{BMV24} shifts the solution by $1$ in time.

\begin{proof}[Proof of Proposition \ref{prop: Sharpness}]
	Let $p < 2_\star$ and set $q \coloneqq p' > 2^\star$. It suffices to find $w \in \EE$ and an operator $\H$ such that 
    \begin{align}
    \label{eq: Mooney sharpness goal}
    (1+ \H)w \in \L^q(\R^{n+1}) \cap \L^2(\R^{n+1})
    \end{align}
    but $w \notin \L^q(\R^{n+1})$. Indeed, by duality this implies that $(1+\H^\star)^{-1}$ does not map $\L^p(\R^{n+1}) \cap \L^2(\R^{n+1})$ into itself and hence we must have $p \leq p_-(\H^\star)$. Interchanging the roles of $\H$ and $\H^\star$ yields the desired counterexample.
    
    We use Mooney's operator and pick $\mu\in[0,\nicefrac{n}{2})$ such that \eqref{eq: Mooney-notLp} holds. We set $v(x,t)\coloneqq \e^{-t}u(x,t)$, which is a weak solution to $v + \partial_t v - \Div_x(A \nabla_x v) =0$ on $\R^{n+1}$. Next, we pick cutoff functions $\psi_1 \in \Cont_\c^\infty(B(0,1))$ and $\psi_2 \in \Cont_\c^\infty((-1,1))$ with $\psi_1(x) = \psi_2(t) = 1$ for $|x|, |t| \leq \nicefrac{1}{2}$. With these choices, we set $\eta(x,t) \coloneqq \psi_1(x) \psi_2(t)$. Then, $w \coloneqq \eta v$ is a weak solution to
	\begin{align}
    \label{eq: Mooney sharpness goal weak}
		w + \partial_t w - \Div_x(A \nabla_x w) &= f,
    \end{align}
    where the right-hand side is given explicitly as
    \begin{align*}
		f &= (\partial_t \eta) v - A\nabla_x v \cdot \nabla_x \eta -\Div_x(A(\nabla_x {\eta}) v)\\
        &=(\partial_t \eta) v - A\nabla_x v \cdot \nabla_x \eta
        - \sum_{i=1}^n \langle\partial_i A_i, \nabla_x \eta \rangle  v - \sum_{i=1}^n \langle A_i, \partial_i \nabla_x \eta \rangle v - A \nabla_x \eta \cdot \nabla_x v
	\end{align*}
    and $A_i$ are the rows of $A$.
    
	It is immediate that $f$ is compactly supported away from the origin in space-time and that all terms on the right-hand side but the third one are bounded. If $(x,t)$ belongs to the support of the third term, then $\nicefrac{1}{2} \leq |x| \leq 1$ and we distinguish three cases for $t$ to see that also the third term is uniformly bounded: First, for $t >0$ the coefficients of $A$ are constant and the term vanishes. Second, for $t\leq -\nicefrac{1}{4}$ we obtain a uniform bound by continuity. Third, for $- \nicefrac{1}{4}< t < 0$ we have $0<(-t)^{\nicefrac{1}{2}} \leq \abs{x}$ and thus \eqref{eq: Mooney-bounds} yields
	\begin{align*}
		\bigg\vert \frac{\d}{\d x_i} A((-t)^{-\nicefrac{1}{2}} x))\bigg\vert = \bigg\vert(\partial_i A) \big((-t)^{-\nicefrac{1}{2}}x\big) (-t)^{-\nicefrac{1}{2}}\bigg\vert \leq C \abs{x}^{-1} \leq 2C.
	\end{align*}
	   Altogether, $f$ in \eqref{eq: Mooney sharpness goal weak} belongs to $\L^r(\R^{n+1})$ for every $r \in [1,\infty]$ but $w$ does not lie in $\L^q(\R^{n+1})$.

    In order to interpret \eqref{eq: Mooney sharpness goal weak} as \eqref{eq: Mooney sharpness goal}, we only need to make sure that $w$ lies in the energy space $\EE$. However, this is an automatic feature of global weak solutions since we have $f \in \L^{2_\star}(\R^{n+1})$, see, e.g., \cite[Proposition~3.1]{ABES2/2019}.
\end{proof}

\section{Open problems}\label{S10}

We conclude this paper by presenting several open problems that naturally arise from our results and merit further investigation in the context of parabolic Riesz transforms.

\begin{enumerate}[(1.)]
    \item \textbf{Weak type estimates at the endpoints:} Suppose that the coefficients of $A$ are real. Is $\D\mathcal{H}^{-1/2}$ of weak type (1,1), that is, does Theorem~\ref{thm: weak-type (1,1)} hold for the full Riesz transform?
    \item \textbf{Case $\boldsymbol{p>2}$:} Study boundedness of parabolic Riesz transforms for $p>2$. The natural conjecture is that the extrapolation range is limited from above by 
    \begin{align*}
        q_+(\mathcal{H})\coloneqq\sup\left\{ p\ge1: \ \left (\lambda \mathbb{D} \mathcal{E}_\lambda \right )_{\lambda>0} \ \text{is} \ \L^p \ \text{bounded} \right\}
    \end{align*}
    as in the elliptic setting of \cite{auscher2007necessary,AEbook2023}.
    \item \textbf{Reverse inequalities:} Find the range of exponents $p$ for which the reverse Riesz transform estimate $\|\H^{1/2}u\|_{p}\lesssim \|\mathbb{D}u\|_{p}$ holds. Let us note that combining part~(2) of Theorem~\ref{thm: Théorème principal} and Lemma~\ref{lem: many functions}, duality yields:
    \begin{cor}
    If $p\in [2,p_-(\H^\star)')$, then
    \begin{equation*}
        \norm{\H^{\nicefrac{1}{2}}u}_p \lesssim \norm{\D u}_p, \qquad u \in \Cont^\infty_0(\R^{n+1}).
    \end{equation*}
    \end{cor}
    \item \textbf{Sharpness in dimension $\boldsymbol{n=1}$:} Prove or disprove that $p_-(\H)=1$ holds in dimension $n=1$. A related question would be to find irregular weak solutions as in \cite{Mooney2021} to equations with complex coefficients in dimension $n=1$. 
    \item \textbf{Domains:} Study Riesz transforms on cylindrical domains. Here even $p=2$ is open, except under some additional smoothness assumptions in $t$~\cite{Ouhabaz21}. For the elliptic setting, see~\cite{BEH20}.
\end{enumerate}
Let us mention that we may also consider degenerate elliptic parts, with degeneracy governed by a spatial Muckenhoupt weight in the class $A_2(\R^n)$. In this case, the parabolic Kato square root estimate has been shown in \cite{AEN2025}. In a work in preparation \cite{Baa2026degenerate}, the first author proves boundedness results for the degenerate parabolic Riesz transform in the range $p\le 2$.

\appendix

\section{Proof of the parabolic Sobolev embedding}\label{annexe}

Here, we give a proof of Lemma \ref{lem: Sobolev}. Our first goal is to derive a representation for functions $u \in \EE$ that is similar to the one obtained in the elliptic setting by means of Riesz potentials. We denote the Fourier transform of a function $f \in \L^2(\R^{n+1})$ by $\hat{f}$ and its inverse transform by $\check{f}$ and collect a few auxiliary results:
\begin{enumerate}[(1.)]
	\item We first note that 
	\begin{align*}
				\mathcal{S}_0 \coloneqq \{u \in \mathcal{S}(\R^{n+1}): 0 \notin \supp(\widehat{u})\}
	\end{align*}
	is dense in $\EE$. Indeed, let $u \in \EE$. There is a sequence $(u_k)_k$ in $\mathcal{S}(\R^{n+1})$ with $u_k \to u$ in $\EE$. We pick smooth cutoff functions $0\le \theta_k \le 1$ such that $\theta_k = 1$ on $\R^{n+1}\setminus B(0,\frac{1}{k})$ and $0$ on $B(0,\frac{1}{2k})$ and put $v_k \coloneqq (\theta_k \widehat{u_k})^\vee$.
      
    We have $v_k \in \mathcal{S}_0$ and $\widehat{v_k} = \theta_k \widehat{u_k} \to \widehat{u}$ in $\L^2(\R^{n+1})$ by dominated convergence and so
	\begin{align*}
	\widehat{\nabla_x v_k} = \ii \xi \widehat{v_k} &\to \ii \xi \widehat{u} = \widehat{\nabla_x u},\\
	\widehat{D_t^{\nicefrac{1}{2}} v_k} = \abs{\tau}^{\nicefrac{1}{2}} \widehat{v_k} &\to \abs{\tau}^{\nicefrac{1}{2}} \widehat{u} = \widehat{D_t^{\nicefrac{1}{2}} u}
	\end{align*}
	in $\L^2(\R^{n+1})$ as well. Plancherel's theorem yields $v_k \to u$ in $\EE$.
	\item The function $(\ii \tau + \abs{\xi}^2)^{-\frac{1}{2}}$ is the Fourier transform of some other function $h$ such that its associated convolution operator is $\L^p-\L^{p^\star}$ bounded whenever $p \in (1,\infty)$, see \cite[Theorem 3.1]{GopalaRao}.
	\item The Fourier multiplier $T$ on $\R^{n+1}$ associated to
	\begin{align*}
		m \coloneqq \frac{(\ii \tau + \abs{\xi}^2)^\frac{1}{2}}{\abs{\tau}^\frac{1}{2} + \ii \abs{\xi}}
	\end{align*}
	is $\L^p$ bounded for all $p \in (1,\infty)$ by the Marcinkiewicz multiplier Theorem, see, e.g., \cite[Corollary 6.2.5]{GrafakosBook2014}.
\end{enumerate}
Now let $u \in \EE$ and pick a sequence $(u_k)_k$ in $\mathcal{S}_0$ with $u_k \to u$ in $\EE$. As $(\ii \tau + \abs{\xi}^2)^\frac{1}{2} \widehat{u_k}$ is again in $\mathcal{S}_0$, we obtain the representation
\begin{align*}
u_k &= \Bigl((\ii \tau + \abs{\xi}^2)^{-\frac{1}{2}}\ (\ii \tau + \abs{\xi}^2)^\frac{1}{2} \widehat{u_k} \Bigr)^\vee = h \ast \Bigl((\ii \tau + \abs{\xi}^2)^\frac{1}{2} \widehat{u_k} \Bigr)^\vee \\
		&= h \ast  \Biggl(\frac{(\ii \tau + \abs{\xi}^2)^\frac{1}{2}}{\abs{\tau}^\frac{1}{2} + \ii \abs{\xi}} \biggl(\abs{\tau}^\frac{1}{2} + \frac{\xi}{|\xi|} \cdot \ii \xi\biggr)\ \widehat{u_k}\Biggr)^\vee 
		=  h \ast T (D_t^{\nicefrac{1}{2}} u_k + R \nabla_x u_k),
\end{align*}
where $R$ is the classical Riesz transform in the $x$-variable.
Since $u_k \to u$ in $\EE$, we see that the left-hand side converges to $u$ in $\L^2(\R^{n+1})$, whereas the right-hand side converges to $h \ast T (D_t^{\nicefrac{1}{2}} u + R \nabla_x u)$ in $\L^{2^\star}(\R^{n+1})$. Consequently, 
\begin{align*}
    u = h \ast T(D_t^{\nicefrac{1}{2}} u + R \nabla_x u).
\end{align*}
Lemma~\ref{lem: Sobolev} follows since $\boldsymbol{f} \mapsto h \ast T \boldsymbol{f}$ is bounded from $\L^p(\R^{n+1})$ to $\L^{p^\star}(\R^{n+1})$ and $R$ is bounded on $\L^p(\R^{n+1})$. %
\nobreak\makebox[\linegoal][r]{\qedsymbol}
	
\subsubsection*{\textbf{Copyright}}

A CC-BY 4.0 \url{https://creativecommons.org/licenses/by/4.0/} public copyright license has been applied by the authors to the present document and will be applied to all subsequent versions up to the Author Accepted Manuscript arising from this submission.

\bibliographystyle{abbrv}
\bibliography{references.bib}

\end{document}